\documentclass[11pt]{amsart}
\usepackage{geometry,ulem}
\usepackage{amscd,amssymb,verbatim,xcolor,mathrsfs}
\usepackage{longtable, multirow}
\usepackage{soul,cancel}
\usepackage{apptools}
\usepackage{xcolor}
\usepackage{tikz} 
\usepackage{longtable}
\usetikzlibrary {arrows.meta,bending}
\date{\today}

\newcommand\al{\alpha}

\newcommand\la{{\lambda}}

\newcommand\pmat{\begin{pmatrix}}
\newcommand\epmat{\end{pmatrix}}

\newcommand\bC{{\mathbb C}}

\newcommand\fra{{\mathfrak a}}
\newcommand\frb{{\mathfrak b}}

\newcommand\frf{{\mathfrak f}}
\newcommand\frg{{\mathfrak g}}
\newcommand\frh{{\mathfrak h}}
\newcommand\frk{{\mathfrak k}}

\newcommand\frp{{\mathfrak p}}
\newcommand\frq{{\mathfrak q}}
\newcommand\frs{{\mathfrak s}}
\newcommand\frt{{\mathfrak t}}
\newcommand\fru{{\mathfrak u}}

\newcommand\bbC{{\mathbb C}}

\newtheorem{theorem}{Theorem}[section]

\newtheorem{example}[theorem]{Example}

\newtheorem{proposition}[theorem]{Proposition}
\newtheorem{remark}[theorem]{Remark}

\newcommand\Ad{{\operatorname{Ad}}}
\newcommand\ad{{\operatorname{ad}}}

\begin{document}
\title[Dirac series for complex $E_8$]{Dirac series for complex $E_8$}
\author{Dan Barbasch}
\author{Kayue Daniel Wong}

\address[Barbasch]{Department of Mathematics, Cornell University, Ithaca, NY 14853,
U.S.A.}
\email{barbasch@math.cornell.edu}

\address[Wong]{School of Science and Engineering, The Chinese University of Hong Kong (Shenzhen), Longgang, Shenzhen, Guangdong 518172, P.R. China}
\email{kayue.wong@gmail.com}

\begin{abstract}
In this paper, we classify all unitary representations with non-zero Dirac cohomology 
 (Dirac series) for complex Lie group of Type $E_8$. This completes the classification of Dirac series
for all complex simple Lie groups.
\end{abstract}

\maketitle
\setcounter{tocdepth}{1}

\section{Introduction}\label{sec:intro}
The notion of Dirac operator plays an important role in representation theory
of real reductive groups. In the 1970s, Parthasarathy \cite{P1, P2} and Schmid used
Dirac operators to give a geometric realization of the discrete series. 
Later, Vogan in \cite{V2} introduced the notion of {\bf Dirac cohomology} for irreducible
representations in order to find sharper estimates for the spectral gap for
locally symmetric spaces. He  formulated a conjecture on its relationship with
the infinitesimal character of the representation, which was subsequently proven by 
Huang and Pand\v zi\'c in \cite{HP1}. 

\medskip

One application of Dirac cohomology is to have a better understanding of the unitary dual $\widehat{G}$.
Indeed, the set of unitary $(\mathfrak{g},K)-$modules with non-zero Dirac cohomology
(the {\bf Dirac series}) $\widehat{G}^d$ contains a large amount of interesting unitary 
representations.
For instance, it is not hard to show that 
$\widehat{G}^d$ strictly contains all unitary modules with non-zero $(\mathfrak{g},K)-$cohomology. By Salamanca-Riba and Vogan-Zuckerman \cite{SR, VZ}, this implies that $\widehat{G}^d$ contains all $A_\frq(\lambda)$ modules in good range.
Later, it was shown that these modules characterize all unitary modules with strongly regular infinitesimal characters. 
Also, Barbasch and Pand\v zi\'c \cite{BP1, BP2} classified all unipotent representations appearing in $\widehat{G}^d$ for all complex groups and some real reductive groups.

\medskip
On the other hand, for any fixed $G$, Dong et. al. \cite{D3, DD1} reduced the study of $\widehat{G}^d$ to a
finite set called {\bf scattered representations} $\widehat{G}^{sc}$. More precisely, it is shown that
every representation in $\widehat{G}^d$ is either a scattered representation, or it is cohomologically induced from a
representation in $\widehat{M}^d$ for some proper theta-stable Levi subgroup $M$ of $G$ in the weakly good range. Using their results,
$\widehat{G}^d$ can be fully classified for all complex classical groups \cite{BDW, DW1, DW2}, $GL(n,\mathbb{R})$ \cite{DW4}, and for all real and complex exceptional groups except for split and complex $E_8$ (\cite{D2}, \cite{DD1}--\cite{DDY}, \cite{DW3}).

\medskip
In this manuscript, we finish the classification of $\widehat{G}^d$ for all complex simple groups 
by studying $\widehat{G}^d$ for Type $E_8$. Although this is a
finite calculation, currently \texttt{atlas} software runs out of
 memory on a usual desktop computer for complex $E_8$. Therefore, we need to take a completely different perspective from that of the aforementioned literature.

\medskip
Instead, we adopt a more conceptual approach by doing standard linear algebra using \texttt{mathematica}. This is a generalization of the techniques in \cite{BDW} and applies to all cases. More explicitly, we use bottom layer
$K-$type arguments to reduce the classification problem to the spherical case and a few nonspherical ones. These $K-$types are useful in reducing the problem to some proper Levi components of $G$. When the Levi component is formed of simple factors of classical type, the results in \cite{BDW} apply. For the remaining cases, we use intertwining operators detailed in \cite{B1} and \cite{BC}.

\medskip
As a consequence, we obtain the main result of this manuscript:
\begin{theorem}[Propositions \ref{cor-a}, \ref{cor-d} and \ref{cor-e}]
Let $G$ be the complex Lie group of Type $E_8$. The representations in $\widehat{G}^d$ are the lowest $K-$type subquotients of modules parabolically induced from a unitary character tensored with a unipotent representation with nonzero Dirac cohomology listed in \cite[Theorem 1.1]{DW2} (for Levi subgroups of Type $D$) and \cite[Section 5.7 -- 5.9]{BP1} (for Levi subgroups of Type $E$). 
\end{theorem}

As a consequence, the main result of \cite{BDW} implies that the Dirac cohomology of the representations in $\widehat{G}^d$ consists of a single $\widetilde{K}-$type occurring with multiplicity $\displaystyle{2^{\lfloor \frac{\mathrm{rank}(\mathfrak{g})}{2}\rfloor}}$.

\section{Preliminaries}
\subsection{Complex groups}
Let $G$ be a connected complex simple Lie group viewed as a real Lie group. Fix a Borel subgroup $B$ of $G$
containing a Cartan subgroup $H$. Fix a Cartan
involution $\theta$ of $G$, and write $K:=G^{\theta}$ for the maximal
compact subgroup. Denote by $\frg_0=\frk_0\oplus\frp_0$ the corresponding Cartan
decomposition of the Lie algebra $\frg_0 = Lie(G)$, and by $\frh_0 = \frt_0 \oplus \fra_0$
be the Cartan decomposition of $\mathfrak{h}_0$. We remove the subscripts of the Lie algebras to denote their 
complexifications, and fix the identifications:
\begin{equation}\label{identifications}
\frg\cong \frg_{0} \oplus
\frg_0, \quad
\frh\cong \frh_{0} \oplus
\frh_0, \quad \frt\cong \{(x,-x) : x\in
\frh_{0} \}, \quad \fra \cong\{(x, x) : x\in
\frh_{0} \}.
\end{equation}

Put $
\Delta^{+}(\frg_0, \frh_0)=\Delta(\frb_0, \frh_0)
$ and let $\rho$ be the half sum of roots in $\Delta^{+}(\frg_0, \frh_0)$. Set
$$
\Delta^+(\frg, \frh)=\left(\Delta^+(\frg_0, \frh_0) \times \{0\}\right) \cup \left(\{0\} \times (-\Delta^+(\frg_0, \frh_0))\right),$$
and let $\rho_{\frg}$ be the half sum of roots in $\Delta^+(\frg, \frh)$. The restrictions to $\frt$ of these positive roots are
$$
\Delta^+(\frg, \frt)=\Delta^+(\frk, \frt)\cup\Delta^+(\frp, \frt).
$$
Denote by $\rho_c$ (resp., $\rho_n$) the half-sum of roots in $\Delta^+(\frk, \frt)$ (resp., $\Delta^+(\frp, \frt)$). Using the identifications in \eqref{identifications}, we have that
\begin{equation}\label{half-sums}
\rho_{\frg}=(\rho, -\rho), \quad \rho_c=\rho_n=\rho.
\end{equation}
We may and we will identify a $K-$type ($\widetilde{K}-$type, $\frk-$type, etc) $V_{\frk}(\eta)$ with its highest weight $\eta \in \Delta^{+}(\frk, \frt)$.

Let $(\lambda_{L}, \lambda_{R})\in \frh_0^{*}\times
\frh_0^{*}$ be such that $\lambda_{L}-\lambda_{R}$ is
a weight of a finite dimensional holomorphic representation of $G$.
 We view $(\lambda_L, \lambda_R)$ as a real-linear functional on $\frh$ by \eqref{identifications}, and write $\bbC_{(\lambda_L, \lambda_R)}$ as the character of $H$ with differential $(\lambda_L, \lambda_R)$. By \eqref{identifications} again, we have
$$
\bbC_{(\lambda_L, \lambda_R)}|_{T}=\bbC_{\lambda_L-\lambda_R}, \quad \bbC_{(\lambda_L, \lambda_R)}|_{A}=\bbC_{\lambda_L+\lambda_R}.
$$
Extend $\bbC_{(\lambda_L, \lambda_R)}$ to a character of $B$, and put $$X(\lambda_{L}, \lambda_{R})
:=K\mbox{-finite part of Ind}_{B}^{G}
(
\bbC_{(\lambda_L, \lambda_R)}  \otimes {\bf 1}
) \mbox{ (normalized induction)}.
$$

Using Frobenius reciprocity, the $K-$type with extremal weight $\lambda_{L}-\lambda_{R}$
occurs with multiplicity one in
$X(\lambda_{L}, \lambda_{R})$. Let
$J(\lambda_L,\lambda_R)$ be the unique subquotient of
$X(\lambda_{L}, \lambda_{R})$ containing the
$K-$type $V_{\frk}(\{\lambda_L -\lambda_R\})$ (here $\{\xi\}$ is the unique dominant weight to which $\xi$ is conjugate under the Weyl group action).

\begin{theorem}\label{thm-Zh} {\rm (Zhelobenko \cite{Zh})}
In the above setting, we have that
\begin{itemize}
\item[a)] Every irreducible admissible ($\frg$, $K$)-module is of the form $J(\lambda_L,\lambda_R)$.
\item[b)] Two such modules $J(\lambda_L,\lambda_R)$ and
$J(\lambda_L^{\prime},\lambda_R^{\prime})$ are equivalent if and
only if there exists $w\in W$ such that
$w\lambda_L=\lambda_L^{\prime}$ and $w\lambda_R=\lambda_R^{\prime}$.
\item[c)] $J(\lambda_L, \lambda_R)$ admits a nondegenerate Hermitian form if and only if there exists
$w\in W$ such that $w(\lambda_L-\lambda_R) =\lambda_L-\lambda_R , w(\lambda_L+\lambda_R) = -\overline{(\lambda_L+\lambda_R)}$.
\item[d)] The representation $X(\lambda_{L}, \lambda_{R})$ is tempered if and only if $\lambda_{L}+\lambda_{R}\in i\frh_0^*$. In this case,
$X(\lambda_{L}, \lambda_{R})=J(\lambda_{L}, \lambda_{R})$.
\end{itemize}
\end{theorem}

Note that $J(\lambda_L,\lambda_R)$ has lowest $K-$type $V_{\frk}(\{\lambda_L-\lambda_R\})$ and infinitesimal character the $W\times W$ orbit of $(\lambda_L, \lambda_R)$.

\subsection{Dirac cohomology}
We recall the construction of Dirac operator and Dirac cohomology.
Let $\langle\ ,\rangle$ be an invariant nondegenerate form such that $\langle\ ,\
\rangle\mid_{\frp_0}$ is positive definite, and $\langle\ ,\
\rangle\mid_{\frk_0}$ is negative definite. Fix an orthonormal basis  $Z_1,\dots, Z_n$
 of $\frp_0$. Let
$U(\frg)$ be the universal enveloping algebra of $\frg$, and $C(\frp)$
be the Clifford algebra of $\frp$ with respect to $\langle\,
,\,\rangle$. The {\bf Dirac operator}
$D\in U(\frg)\otimes C(\frp)$ is defined as
$$D=\sum_{i=1}^{n}\, Z_i \otimes Z_i.$$
The operator  $D$ does not depend on the choice of the
orthonormal basis $Z_i$ and  is $K-$invariant for the diagonal
action of $K$ induced by the adjoint actions on both factors.

\medskip
Define  $\Delta: \frk \to U(\frg)\otimes C(\frp)$ by $\Delta(X)=X\otimes
1 + 1\otimes \alpha(X)$, where $\alpha:\frk\to C(\frp)$ is the
composition of $\ad:\frk\longrightarrow \mathfrak{so}(\frp)$ with the embedding  $\mathfrak{so}(\frp)\cong\wedge^2(\frp)\hookrightarrow C(\frp)$. Write $\frk_{\Delta}:=\Delta(\frk)$, and denote by $\Omega_{\frg}$ (resp. $\Omega_{\frk}$)
the Casimir operator of $\frg$ (resp. $\frk$). Let $\Omega_{\frk_{\Delta}}$ be the image of $\Omega_{\frk}$ under $\Delta$. Then \cite{P1} implies that
\begin{equation*}\label{D-square}
D^2=-\Omega_{\frg}\otimes 1 + \Omega_{\frk_{\Delta}} + (\|\rho_c\|^2-\|\rho_{\frg}\|^2) 1\otimes 1,
\end{equation*}
where $\rho_{\frg}$ and $\rho_c$ are the corresponding half sums of positive roots of $\frg$ and $\frk$.

Let
$$
\widetilde{K}:=\{ (k,s)\in K\times {\rm Spin}(\frp_0)\ :\ \Ad (k)=p(s)\},
$$
where $p: \text{Spin}(\frp_0)\rightarrow \text{SO}(\frp_0)$ is the spin double covering map. If $\pi$ is a
($\frg$, $K$)-module, and if $S_G$ denotes a spin module for
$C(\frp)$, then $\pi\otimes S_G$ is a $(U(\frg)\otimes C(\frp),
\widetilde{K})$ module.

The action of $U(\frg)\otimes C(\frp)$ is
the obvious one, and $\widetilde{K}$ acts on both factors; on $\pi$
through $K$ and on $S_G$ through the spin group
$\text{Spin}\,{\frp_0}$.
The Dirac operator acts on $\pi\otimes S_G$. The Dirac
cohomology of $\pi$ is defined as the $\widetilde{K}-$module
\begin{equation*}\label{def-Dirac-cohomology}
H_D(\pi)=\text{Ker}\, D/ (\text{Im} \, D \cap \text{Ker} D).
\end{equation*}

The following foundational result on Dirac cohomology, conjectured
by Vogan,  was proven by Huang and Pand\v zi\'c in 2002:

\begin{theorem}[\cite{HP1} Theorem 2.3]\label{thm-HP}
Let $\pi$ be an irreducible ($\frg$, $K$)-module.
Assume that the Dirac
cohomology of $\pi$ is nonzero, and that it contains the $\widetilde{K}-$type with highest weight $\gamma\in\frt^{*}\subset\frh^{*}$. Then the infinitesimal character of $\pi$ is conjugate to
$\gamma+\rho_{c}$ under $W(\frg,\frh)$.
\end{theorem}

\subsection{Unitary modules with Dirac cohomology}\label{sec-spin}
Let $\pi$ be an irreducible $(\frg,K)-$module for a complex Lie group $G$.
By Theorem \ref{thm-HP} and \eqref{identifications}, $\pi$ has Dirac cohomology
if and only if its Zhelobenko parameter $(w_1\la_L,w_2\la_R)$ satisfies
\begin{equation} \label{eq-HPcomplex}
\begin{cases}
  w_1\la_L-w_2\la_R=\tau +\rho\\
  w_1\la_L +w_2\la_R=0,
\end{cases}
\end{equation}
where $V_{\mathfrak{k}}(\tau)$ is a $\widetilde{K}-$type in $H_D(\pi)$. The second equation implies $\la_R=-w_2^{-1}w_1\la_L.$ Since $\tau
+\rho$ is regular integral, the first equation implies that
\begin{equation} \label{eq-HP}
2\la := 2w_1\la_L = \tau + \rho 
\end{equation}
is regular integral.

Consequently, the module can be written as $\pi = J(\la,-s\la)$ with $2\la$ regular integral,
and the first equation of \eqref{eq-HPcomplex} implies that the {\it only} $\widetilde{K}-$type
that can appear in $H_D(\pi)$ is 
$V_{\mathfrak{k}}(2\la - \rho)$.
Furthermore, if $J(\la,-s\la)$ is Hermitian (e.g. if $J(\la,-s\la)$ is unitary),
it follows as in \cite{BP1} that $s$ is an involution.

\subsection{Bottom Layer $K-$types} \label{sec-bottom}
For the rest of this paper, we use the following Dynkin diagram and simple roots of $E_8$:
\begin{equation} \label{eq-e8}
\begin{tikzpicture}
\draw
    (0,0) node[circle,fill=black,inner sep=0pt,minimum size=3pt,label=below:{$\alpha_1$}] {}
 -- (1,0) node[circle,fill=black,inner sep=0pt,minimum size=3pt,label=below:{$\alpha_3$}] {}
--   (2,0) node[circle,fill=black,inner sep=0pt,minimum size=3pt,label=below:{$\alpha_4$}] {}
 -- (3,0) node[circle,fill=black,inner sep=0pt,minimum size=3pt,label=below:{$\alpha_5$}] {}
-- (4,0) node[circle,fill=black,inner sep=0pt,minimum size=3pt,label=below:{$\alpha_6$}] {}
-- (5,0) node[circle,fill=black,inner sep=0pt,minimum size=3pt,label=below:{$\alpha_7$}] {}
-- (6,0) node[circle,fill=black,inner sep=0pt,minimum size=3pt,label=below:{$\alpha_8$}] {};

\draw
 (2,0.7) node[circle,fill=black,inner sep=0pt,minimum size=3pt,label=above:{$\alpha_2$}] {}
--   (2,0) node {};
\end{tikzpicture}
\end{equation}

By the discussions in Section \ref{sec-spin}, one only focuses on $\pi = J(\la,-s\la)$ where $2\la$ is regular
integral, and $s \in W$ is an involution. Conjugate $\la + s\la$ such that
\begin{equation} \label{eq-etal}
\eta := \{\la + s\la\} = [k_1, k_2, \dots, k_8] := k_1 \omega_1 + k_2 \omega_2 + \dots k_8 \omega_8,
\end{equation}
is a dominant weight (here $k_i \in \mathbb{N}$, and $\omega_i$ are the fundamental weights of $E_8$).

Let $M$ be the Levi subgroup of $G$ determined by the nodes 
\begin{equation} \label{eq-m}
I(M) := \{i\ |\ \langle \alpha_i, \eta \rangle = k_i = 0\}
\end{equation}
of the Dynkin diagram \eqref{eq-e8}. Suppose $M = \mathcal{F}_1 \times \dots \times \mathcal{F}_r \times (\mathbb{C}^*)^s$, 
where each $\mathcal{F}_{\ell}$ is a simple group. Then one can choose $\nu_{\ell} \in \frh_{\ell}^*$ (here $\frh_{\ell}$ is the Cartan subalgebra of $\mathcal{F}_{\ell}$), and a unitary character $\mathbb{C}_{\tau(\eta)}$ of $(\mathbb{C}^*)^s$ such that the induced module
\begin{equation} \label{eq-fl}
\mathrm{Ind}_{MN}^G\left(\bigotimes_{\ell =1}^r J_{\mathcal{F}_{\ell}}(\nu_{\ell},\nu_{\ell}) \otimes \mathbb{C}_{\tau(\eta)} \otimes {\bf 1} \right),
\end{equation}
has the same infinitesimal character and lowest $K-$type as $\pi$. Consequently, $\pi$ appears as the lowest $K-$type
subquotient of \eqref{eq-fl}.
By Theorem \ref{thm-Zh}(b), one can further assume that $2\nu_{\ell}$ is regular and integral
for each spherical module $J_{\mathcal{F}_{\ell}}(\nu_{\ell},\nu_{\ell})$
in \eqref{eq-fl}.

\begin{proposition} \label{prop-bottomlayer}
Let $\pi = J(\la,-s\la)$ be an irreducible Hermitian $(\frg,K)-$module with $2\la$ regular and integral.
Consider the induced module \eqref{eq-fl} corresponding to $\pi$, where $2\nu_{\ell}$ are chosen to be regular and integral for all $1 \leq \ell \leq r$. Suppose 
$$\langle 2\nu_{\ell}, \alpha \rangle \notin \{1, 2\},$$
for some simple root $\alpha$ corresponding to $\mathcal{F}_{\ell}$, then $\pi$ is not unitary. 
\end{proposition}
\begin{proof}
By the hypothesis of $2\nu_{\ell}$, and the deformation arguments in
\cite{BDW} (which can be generalized to all complex groups),
$J_{\mathcal{F}_{\ell}}(\nu_{\ell},\nu_{\ell})$ is not unitary on the
level of the adjoint $(\mathcal{F}_{\ell} \cap K)-$type $V_{\frf_{\ell} \cap \frk}(\delta)$. 

We claim that for all possible simple components $\mathcal{F}_{\ell}$ of a Levi subgroup of $E_8$, $\eta + \delta$ is {\bf $M-$bottom layer} for all $\eta$ of the form given in \eqref{eq-etal}. Indeed, the adjoint representation for each $\mathcal{F}_{\ell}$ can be written as a sum of simple roots by:
\begin{center}
\begin{longtable}{|c|c|c|c|}
\hline
Type & Dynkin diagram & Highest weight $\delta$ of adjoint representation \\
\hline \hline
$A_p$ & 
\begin{tikzpicture}
\draw
    (0,0) node[circle,fill=black,inner sep=0pt,minimum size=3pt,label=below:{$\beta_1$}] {}
 -- (1,0) node[circle,fill=black,inner sep=0pt,minimum size=3pt,label=below:{$\beta_2$}] {};
 
\draw (1.5,0) node {$\cdots$};

\draw
    (2,0) node[circle,fill=black,inner sep=0pt,minimum size=3pt,label=below:{$\beta_{p-1}$}] {}
 -- (3,0) node[circle,fill=black,inner sep=0pt,minimum size=3pt,label=below:{$\beta_p$}] {};
\end{tikzpicture} &
$\beta_1 + \beta_2 + \dots + \beta_{p-1} + \beta_p$\\

\hline
$D_p$ & 
\begin{tikzpicture}
\draw
    (-1,0) node[circle,fill=black,inner sep=0pt,minimum size=3pt,label=below:{$\beta_1$}] {}
 -- (0,0) node[circle,fill=black,inner sep=0pt,minimum size=3pt,label=below:{$\beta_2$}] {};
 
\draw (0.5,0) node {$\cdots$};

\draw
    (1,0) node[circle,fill=black,inner sep=0pt,minimum size=3pt,label=below:{$\beta_{p-3}$}] {}
 -- (2,0) node {};

\draw
 (2,0.7) node[circle,fill=black,inner sep=0pt,minimum size=3pt,label=above:{$\beta_{p-1}$}] {}
--  (2,0) node[circle,fill=black,inner sep=0pt,minimum size=3pt,label=below:{$\beta_{p-2}$}] {}
 -- (3,0) node[circle,fill=black,inner sep=0pt,minimum size=3pt,label=below:{$\beta_p$}] {};
\end{tikzpicture} &
$\beta_1 + 2\beta_2 + \dots + 2\beta_{p-2} + \beta_{p-1} + \beta_p$\\

\hline
$E_6$ & 
\begin{tikzpicture}
\draw
    (0,0) node[circle,fill=black,inner sep=0pt,minimum size=3pt,label=below:{$\beta_1$}] {}
 -- (1,0) node[circle,fill=black,inner sep=0pt,minimum size=3pt,label=below:{$\beta_3$}] {}
--   (2,0) node[circle,fill=black,inner sep=0pt,minimum size=3pt,label=below:{$\beta_4$}] {}
 -- (3,0) node[circle,fill=black,inner sep=0pt,minimum size=3pt,label=below:{$\beta_5$}] {}
-- (4,0) node[circle,fill=black,inner sep=0pt,minimum size=3pt,label=below:{$\beta_6$}] {};

\draw
 (2,0.7) node[circle,fill=black,inner sep=0pt,minimum size=3pt,label=above:{$\beta_2$}] {}
--   (2,0) node {};
\end{tikzpicture} &
$\beta_1 + 2\beta_2 + 2\beta_3 + 3\beta_4 + 2\beta_5 + \beta_6$\\

\hline
$E_7$ & 
\begin{tikzpicture}
\draw
    (0,0) node[circle,fill=black,inner sep=0pt,minimum size=3pt,label=below:{$\beta_1$}] {}
 -- (1,0) node[circle,fill=black,inner sep=0pt,minimum size=3pt,label=below:{$\beta_3$}] {}
--   (2,0) node[circle,fill=black,inner sep=0pt,minimum size=3pt,label=below:{$\beta_4$}] {}
 -- (3,0) node[circle,fill=black,inner sep=0pt,minimum size=3pt,label=below:{$\beta_5$}] {}
-- (4,0) node[circle,fill=black,inner sep=0pt,minimum size=3pt,label=below:{$\beta_6$}] {}
-- (5,0) node[circle,fill=black,inner sep=0pt,minimum size=3pt,label=below:{$\beta_7$}] {};

\draw
 (2,0.7) node[circle,fill=black,inner sep=0pt,minimum size=3pt,label=above:{$\beta_2$}] {}
--   (2,0) node {};
\end{tikzpicture}  &
$2\beta_1 + 2\beta_2 + 3\beta_3 + 4\beta_4 + 3\beta_5 + 2\beta_6 + \beta_7$\\
\hline
\end{longtable}
\end{center}

It is straightforward  that 
\begin{enumerate}
\item[(a)] \fbox{$\langle \eta + \delta, \alpha_i \rangle = \langle \delta, \alpha_i \rangle \geq 0$ for all $i \in I(M)$}, since $\eta$ does not contain any fundamental $K-$types $\omega_i$, and $\delta$ is $M-$dominant weight.
\item[(b)] \fbox{$\langle \eta + \delta, \alpha_j \rangle = k_j + \langle \delta, \alpha_j \rangle \geq k_j - 1 \geq 0$ for all $j \notin I(M)$}.
\end{enumerate}

The $k_j$ are from \eqref{eq-etal}, and the first $\geq$ in (b) follows from the fact that the Dynkin diagram is simply laced, so that 
$\langle \alpha_i, \alpha_j \rangle = 0$ or $-1$ (with $-1$ whenever $\alpha_i$, $\alpha_j$ are linked in the diagram).

\medskip
For example, let $M$ be of Type $E_7$, so that the $\beta_i$ in the above table match with $\alpha_i$
in \eqref{eq-e8} { for $1\le i\le 7.$} Then $\eta = k\omega_8$ for $k \geq 1$, and
$$\eta + \delta = k\omega_8 + 2\alpha_1 + 2\alpha_2 + 3\alpha_3 + 4\alpha_4 + 3\alpha_5 + 2\alpha_6 + \alpha_7.$$
Therefore, $\langle \eta + \delta, \alpha_8\rangle = \langle k\omega_8 + \alpha_7, \alpha_8\rangle = k + \langle \alpha_7, \alpha_8\rangle = k-1 \geq 0$ for all $k$. In other words, $\eta + \delta$ is always $M-$bottom layer.

Consequently, this implies that the form on the induced module \eqref{eq-fl} as well as
$\pi$ is indefinite on $\eta$ and $\eta+\delta$; the result follows.
\end{proof}

\subsection{General strategy} \label{sec-strategy}
We summarize our strategy of determining $\widehat{G}^d$ for complex
$E_8$.
By Theorem \ref{thm-HP}, we focus on irreducible modules $\pi = J(\la_L,\la_R)$ such that $2\la_L$ is regular and integral.
\begin{itemize}
\item[(i)] As described in the
previous section, the lowest $K-$type $\eta$ of $\pi$ determines a Levi subgroup $M = \mathcal{F}_1 \times$ $\dots$ $\times$
$\mathcal{F}_r \times (\mathbb{C}^*)^s$ of $G$, so that $\pi$ is the lowest $K-$type subquotient of the induced module
\begin{equation}
  \label{eq:induced}
\mathrm{Ind}_{MN}^G\left(\pi_{sph} \otimes \mathbb{C}_{\tau(\eta)}
  \otimes {\bf 1} \right), \quad \pi_{sph} := \bigotimes_{\ell = 1}^r
J_{\mathcal{F}_{\ell}}(\nu_{\ell},\nu_{\ell}). 
\end{equation}
\item[(ii)] By Proposition \ref{prop-bottomlayer}, in order for $\pi$
  to be unitary, the spherical module $\pi_{sph}$ must satisfy
\begin{center}
$\langle \nu_{\ell}, \alpha \rangle =\frac{1}{2}$ or $1$ \end{center}
for all simple roots $\alpha$ of $\mathcal{F}_{\ell}$. In other words, there are finitely many candidates of $\pi_{sph}$. 
\item[(iii)] If $\pi_{sph}$ is unitary, then $\pi$ is automatically  unitary.
  Otherwise there are finitely many {\it non-unitarity certificates}. When $\eta$ is
  large enough, these certificates turn out to
  be bottom layer, so we are left with a finite set of $\eta$ that have to be dealt
  with differently. In most cases, the induced module \eqref{eq:induced} with these (finite) choices of $\eta$ and $\nu_{\ell}$ have infinitesimal characters whose
  $2\lambda_L-$values are either singular or non-integral, i.e. Equation \eqref{eq-HP} is violated. For instance, this is always the case when $M$ consists only of Type $A$ factors (see Section \ref{sec-A} below). 
  
  \item[(iv)] For the very few remaining cases, 
  we enlarge the Levi component, and embed the
  irreducible representation as the lowest $K-$type component in an
  induced module as in Equation \eqref{eq:induced}. We will show that either the 
  representations to be induced are unitary, or else we find a
  ``smallest''   $\eta+\sigma$ with signature opposite to $\eta$ that is bottom layer.
  \end{itemize}

\section{Type $A$ Levi subgroups} \label{sec-A}
This section is devoted to dealing with the cases when $M$ consists only of Type $A$ simple factors
$\mathcal{F}_1$, $\dots$, $\mathcal{F}_r$. We will study each $J_{\mathcal{F}_{\ell}}(\nu_{\ell},\nu_{\ell})$ in Step (i) of Section \ref{sec-strategy}
individually, where $\frf_{\ell} = \mathfrak{sl}(p+1,\mathbb{C})$ is of Type $A_p$ with simple roots:
\begin{center}
\begin{tikzpicture}
\draw
    (0,0) node[circle,fill=black,inner sep=0pt,minimum size=3pt,label=below:{$\beta_1$}]{}
 -- (1,0) node[circle,fill=black,inner sep=0pt,minimum size=3pt,label=below:{$\beta_2$}] {};
 --(1,0)
\draw (1.5,0) node {$\cdots$};

\draw
    (2,0) node[circle,fill=black,inner sep=0pt,minimum size=3pt,label=below:{$\beta_{p-1}$}] {}
 -- (3,0) node[circle,fill=black,inner sep=0pt,minimum size=3pt,label=below:{$\beta_p$}] {};
\end{tikzpicture}
\end{center}

Based on Equation \eqref{eq-HP} and the discussion that follows, if an irreducible unitary representation $\pi$ has nonzero cohomology, then $2\lambda$ must be regular and integral. For Levi subgroups of type $A$, we will see that this forces $J_{\mathcal{F}_{\ell}}(\nu_{\ell},\nu_{\ell})$ to be trivial, implying that $\pi$ is induced from unitary characters of type $A$. In summary, we will prove the following:
\begin{proposition} \label{cor-a}
Let $\pi = J(\lambda,-s\lambda) \in \widehat{G}$ be such that $2\lambda$ is regular integral (e.g. $\pi \in \widehat{G}^d$), and 
the highest weight of its lowest $K-$type $\eta = \{\lambda + s\lambda\}$ defines a Levi subgroup 
$M$ (c.f. \eqref{eq-m}) consisting only of Type $A$ simple factors. Then $\pi$ must be the lowest $K-$type subquotient of the unitarily induced module
$$\mathrm{Ind}_{M'N'}^{E_8}\left(\mathrm{triv}_A \otimes \mathbb{C}_{\tau'(\eta)} \otimes {\bf 1}\right)$$
for some Levi subgroup $M' \leq M$, and $\mathrm{triv}_A$ is the trivial representation of the Type $A$ simple factors in $M'$.
\end{proposition}
\begin{proof}
By \cite{V1}, the $(\mathcal{F}_{\ell} \cap K)-$types used to detect non-unitarity of $(\frf_{\ell}, \mathcal{F}_{\ell} \cap K)-$modules 
(non-unitarity certificates) have highest weights
\begin{equation} \label{eq-beta}
\begin{aligned}
&\sigma_1 := \beta_1 + \beta_2 +\ldots + \beta_{p-1} + \beta_p, \\ 
&\sigma_2 := \beta_1 + 2(\beta_2 +\ldots + \beta_{p-1}) + \beta_p,\\
&\sigma_3 := \beta_1 + 2\beta_2 + 3(\beta_3 + \ldots + \beta_{p-2}) +
2\beta_{p-1} + \beta_p, \\
&\cdots
\end{aligned}
\end{equation}  

For each Levi subgroup $M$ of $G$, we mark the nodes of $I(M)$ on the Dynkin diagram \eqref{eq-e8} by $\Delta.$ For example, if $I(M) = \{1,3,4,6,7\}$, one has:
\begin{center}
\begin{tikzpicture}
\draw
    (0,0) node[circle,inner sep=0pt,minimum size=3pt,label=below:{$\alpha_1$}] {$\Delta$}
 -- (1,0) node[circle,inner sep=0pt,minimum size=3pt,label=below:{$\alpha_3$}] {$\Delta$}
--   (2,0) node[circle,inner sep=0pt,minimum size=3pt,label=below:{$\alpha_4$}] {$\Delta$}
 -- (3,0) node[circle,fill=black,inner sep=0pt,minimum size=3pt,label=below:{$\alpha_5$}] {}
-- (4,0) node[circle,inner sep=0pt,minimum size=3pt,label=below:{$\alpha_6$}] {$\Delta$}
-- (5,0) node[circle,inner sep=0pt,minimum size=3pt,label=below:{$\alpha_7$}] {$\Delta$}
-- (6,0) node[circle,fill=black,inner sep=0pt,minimum size=3pt,label=below:{$\alpha_8$}] {};

\draw
 (2,0.7) node[circle,fill=black,inner sep=0pt,minimum size=3pt,label=above:{$\alpha_2$}] {}
--   (2,0) node {};
\end{tikzpicture}
\end{center}
Suppose all
  $\al_j\notin I(M)$ are solely connected to the endpoints of the diagram of
  $I(M)$ (the diagram above is one such example), the proof of Proposition
\ref{prop-bottomlayer} shows that the $K-$types with highest weights $\eta + \sigma_i$
are $M-$bottom layer for all $\sigma_i$ in \eqref{eq-beta}.
In other words, $\pi$ is unitary if and only if $\pi_{sph}$ in Step (i) of Section \ref{sec-strategy} is unitary in these cases.
Furthermore, by the main result of \cite{BDW}, $\pi \in \widehat{G}^d$
if and only if $\pi_{sph}$ is parabolically induced from
the trivial module. By induction in stages, this implies $\pi$ is a subquotient of a parabolically induced module 
from a unitary character. 

This reduces the
  proof to the cases when $M$ has a factor
$\mathcal{F}_{\ell}$ of Type $A_p$ such that the  nodes $\alpha_j$ with $j
\notin I(\mathcal{F}_{\ell})$ are only connected to nodes inner to the
diagram of the Type $A-$factor. The next list records these cases, and
for what values the $\mu+\sigma_j$ are bottom layer. The nodes in
$I(M)$ are marked $\Delta,$ the relevant nodes not in $I(M)$ are
marked $\times$. The remaining nodes are marked $\bullet$ and do not
figure in the arguments. The $k_j$ record the lowest value for which
the $K-$type is bottom layer.

\begin{center}
\begin{longtable}{|c|c|c|c|}
\hline
Type & $I(\mathcal{F}_{\ell})$ (in $\Delta$) and $j$ (in $\times$) &  Bottom layer $\sigma + \eta$\\
\hline \hline
(i) & 
\begin{tikzpicture}
\draw
    (0,-1) node[circle,inner sep=0pt,minimum size=3pt,label=below:{$\alpha_1$}] {$\Delta$}
 -- (1,-1) node[circle,inner sep=0pt,minimum size=3pt,label=below:{$\alpha_3$}] {$\Delta$}
--   (2,-1) node[circle,inner sep=0pt,minimum size=3pt,label=below:{$\alpha_4$}] {$\Delta$}
 -- (3,-1) node[circle,inner sep=0pt,minimum size=3pt,label=below:{$\alpha_5$}] {$\Delta$}
-- (4,-1) node[circle,inner sep=0pt,minimum size=3pt,label=below:{$\alpha_6$}] {$\Delta$}
-- (5,-1) node[circle,inner sep=0pt,minimum size=3pt,label=below:{$\alpha_7$}] {$\Delta$}
-- (6,-1) node[circle,inner sep=0pt,minimum size=3pt,label=below:{$\alpha_8$}] {$\Delta$};

\draw
 (2,-0.3) node[circle,inner sep=0pt,minimum size=3pt,label=above:{$\alpha_2$}] {$\times$}
--   (2,-1) node {};
\end{tikzpicture} &
$\begin{matrix} \sigma_1 + k_2\omega_2 &(k_2 \geq 1),\\
\sigma_2 + k_2\omega_2 &(k_2 \geq 2),\\
\sigma_3 + k_2\omega_2 &(k_2 \geq 3),\\
\sigma_4 + k_2\omega_2 &(k_2 \geq 3)
\\
\\
\\
\end{matrix}$
\\

\hline
(ii) & 
\begin{tikzpicture}
\draw
    (0,-1) node[circle,inner sep=0pt,minimum size=3pt,label=below:{$\alpha_1$}] {$\Delta$}
 -- (1,-1) node[circle,inner sep=0pt,minimum size=3pt,label=below:{$\alpha_3$}] {$\Delta$}
--   (2,-1) node[circle,inner sep=0pt,minimum size=3pt,label=below:{$\alpha_4$}] {$\Delta$}
 -- (3,-1) node[circle,inner sep=0pt,minimum size=3pt,label=below:{$\alpha_5$}] {$\Delta$}
-- (4,-1) node[circle,inner sep=0pt,minimum size=3pt,label=below:{$\alpha_6$}] {$\Delta$}
-- (5,-1) node[circle,inner sep=0pt,minimum size=3pt,label=below:{$\alpha_7$}] {$\Delta$}
-- (6,-1) node[circle,fill=black, inner sep=0pt,minimum size=3pt,label=below:{$\alpha_8$}] {};

\draw
 (2,-0.3) node[circle,inner sep=0pt,minimum size=3pt,label=above:{$\alpha_2$}] {$\times$}
--   (2,-1) node {};
\end{tikzpicture} &
$\begin{matrix} \sigma_1  + k_2\omega_2 &(k_2 \geq 1),\\
\sigma_2  + k_2\omega_2 &(k_2 \geq 2),\\
\sigma_3  + k_2\omega_2 &(k_2 \geq 3)
\\
\\
\\
\end{matrix}$

\\

\hline
(iii) & 
\begin{tikzpicture}
\draw
    (0,-1) node[circle,fill=black,inner sep=0pt,minimum size=3pt,label=below:{$\alpha_1$}] {}
 -- (1,-1) node[circle,inner sep=0pt,minimum size=3pt,label=below:{$\alpha_3$}] {$\Delta$}
--   (2,-1) node[circle,inner sep=0pt,minimum size=3pt,label=below:{$\alpha_4$}] {$\Delta$}
 -- (3,-1) node[circle,inner sep=0pt,minimum size=3pt,label=below:{$\alpha_5$}] {$\Delta$}
-- (4,-1) node[circle,inner sep=0pt,minimum size=3pt,label=below:{$\alpha_6$}] {$\Delta$}
-- (5,-1) node[circle,inner sep=0pt,minimum size=3pt,label=below:{$\alpha_7$}] {$\Delta$}
-- (6,-1) node[circle,inner sep=0pt,minimum size=3pt,label=below:{$\alpha_8$}] {$\Delta$};

\draw
 (2,-0.3) node[circle,inner sep=0pt,minimum size=3pt,label=above:{$\alpha_2$}] {$\times$}
--   (2,-1) node {};
\end{tikzpicture} &
$\begin{matrix}  \sigma_1  + k_2\omega_2 &(k_2 \geq 1),\\
\sigma_2  + k_2\omega_2 &(k_2 \geq 2),\\
\sigma_3 + k_2\omega_2 &(k_2 \geq 2)
\\
\\
\\
\end{matrix}$\\

\hline
(iii)' & 
\begin{tikzpicture}
\draw
    (0,-1) node[circle,fill=black,inner sep=0pt,minimum size=3pt,label=below:{$\alpha_1$}] {}
 -- (1,-1) node[circle,inner sep=0pt,minimum size=3pt,label=below:{$\alpha_3$}] {$\times$}
--   (2,-1) node[circle,inner sep=0pt,minimum size=3pt,label=below:{$\alpha_4$}] {$\Delta$}
 -- (3,-1) node[circle,inner sep=0pt,minimum size=3pt,label=below:{$\alpha_5$}] {$\Delta$}
-- (4,-1) node[circle,inner sep=0pt,minimum size=3pt,label=below:{$\alpha_6$}] {$\Delta$}
-- (5,-1) node[circle,inner sep=0pt,minimum size=3pt,label=below:{$\alpha_7$}] {$\Delta$}
-- (6,-1) node[circle,inner sep=0pt,minimum size=3pt,label=below:{$\alpha_8$}] {$\Delta$};

\draw
 (2,-0.3) node[circle,inner sep=0pt,minimum size=3pt,label=above:{$\alpha_2$}] {$\Delta$}
--   (2,-1) node {};
\end{tikzpicture} &
$\begin{matrix}\sigma_1  + k_3\omega_3 &(k_3 \geq 1),\\
\sigma_2  + k_3\omega_3 &(k_3 \geq 2),\\
\sigma_3  + k_3\omega_3 &(k_3 \geq 2)
\\
\\
\\
\end{matrix}$\\

\hline
(iv) & 
\begin{tikzpicture}
\draw
    (0,-1) node[circle,inner sep=0pt,minimum size=3pt,label=below:{$\alpha_1$}] {$\Delta$}
 -- (1,-1) node[circle,inner sep=0pt,minimum size=3pt,label=below:{$\alpha_3$}] {$\Delta$}
--   (2,-1) node[circle,inner sep=0pt,minimum size=3pt,label=below:{$\alpha_4$}] {$\Delta$}
 -- (3,-1) node[circle,inner sep=0pt,minimum size=3pt,label=below:{$\alpha_5$}] {$\Delta$}
-- (4,-1) node[circle,inner sep=0pt,minimum size=3pt,label=below:{$\alpha_6$}] {$\Delta$}
-- (5,-1) node[circle,fill=black,inner sep=0pt,minimum size=3pt,label=below:{$\alpha_7$}] {}
-- (6,-1) node[circle,fill=black, inner sep=0pt,minimum size=3pt,label=below:{$\alpha_8$}] {};

\draw
 (2,-0.3) node[circle,inner sep=0pt,minimum size=3pt,label=above:{$\alpha_2$}] {$\times$}
--   (2,-1) node {};
\end{tikzpicture} &
$\begin{matrix} \sigma_1  + k_2\omega_2 &(k_2 \geq 1),\\
\sigma_2  + k_2\omega_2 &(k_2 \geq 2),\\
\sigma_3  + k_2\omega_2 &(k_2 \geq 3)
\\
\\
\\
\end{matrix}$\\

\hline
(v) & 
\begin{tikzpicture}
\draw
    (0,-1) node[circle,fill=black,inner sep=0pt,minimum size=3pt,label=below:{$\alpha_1$}] {}
 -- (1,-1) node[circle,inner sep=0pt,minimum size=3pt,label=below:{$\alpha_3$}] {$\Delta$}
--   (2,-1) node[circle,inner sep=0pt,minimum size=3pt,label=below:{$\alpha_4$}] {$\Delta$}
 -- (3,-1) node[circle,inner sep=0pt,minimum size=3pt,label=below:{$\alpha_5$}] {$\Delta$}
-- (4,-1) node[circle,inner sep=0pt,minimum size=3pt,label=below:{$\alpha_6$}] {$\Delta$}
-- (5,-1) node[circle,inner sep=0pt,minimum size=3pt,label=below:{$\alpha_7$}] {$\Delta$}
-- (6,-1) node[circle,fill=black, inner sep=0pt,minimum size=3pt,label=below:{$\alpha_8$}] {};

\draw
 (2,-0.3) node[circle,inner sep=0pt,minimum size=3pt,label=above:{$\alpha_2$}] {$\times$}
--   (2,-1) node {};
\end{tikzpicture} &
$\begin{matrix} \sigma_1  + k_2\omega_2 &(k_2 \geq 1),\\
\sigma_2  + k_2\omega_2 &(k_2 \geq 2),\\
\sigma_3  + k_2\omega_2 &(k_2 \geq 2)
\\
\\
\\
\end{matrix}$\\
\hline

(v)' & 
\begin{tikzpicture}
\draw
    (0,-1) node[circle,fill=black,inner sep=0pt,minimum size=3pt,label=below:{$\alpha_1$}] {}
 -- (1,-1) node[circle,,inner sep=0pt,minimum size=3pt,label=below:{$\alpha_3$}] {$\times$}
--   (2,-1) node[circle,inner sep=0pt,minimum size=3pt,label=below:{$\alpha_4$}] {$\Delta$}
 -- (3,-1) node[circle,inner sep=0pt,minimum size=3pt,label=below:{$\alpha_5$}] {$\Delta$}
-- (4,-1) node[circle,inner sep=0pt,minimum size=3pt,label=below:{$\alpha_6$}] {$\Delta$}
-- (5,-1) node[circle,inner sep=0pt,minimum size=3pt,label=below:{$\alpha_7$}] {$\Delta$}
-- (6,-1) node[circle,fill=black, inner sep=0pt,minimum size=3pt,label=below:{$\alpha_8$}] {};

\draw
 (2,-0.3) node[circle,inner sep=0pt,minimum size=3pt,label=above:{$\alpha_2$}] {$\Delta$}
--   (2,-1) node {};
\end{tikzpicture} &
$\begin{matrix} \sigma_1  + k_2\omega_2 &(k_2 \geq 1),\\
\sigma_2  + k_2\omega_2 &(k_2 \geq 2),\\
\sigma_3  + k_2\omega_2 &(k_2 \geq 2)
\\
\\
\\
\end{matrix}$\\
\hline

(vi) & 
\begin{tikzpicture}
\draw
    (0,-1) node[circle,inner sep=0pt,minimum size=3pt,label=below:{$\alpha_1$}] {$\Delta$}
 -- (1,-1) node[circle,,inner sep=0pt,minimum size=3pt,label=below:{$\alpha_3$}] {$\Delta$}
--   (2,-1) node[circle,inner sep=0pt,minimum size=3pt,label=below:{$\alpha_4$}] {$\Delta$}
 -- (3,-1) node[circle,inner sep=0pt,minimum size=3pt,label=below:{$\alpha_5$}] {$\Delta$}
-- (4,-1) node[circle,fill=black,inner sep=0pt,minimum size=3pt,label=below:{$\alpha_6$}] {}
-- (5,-1) node[circle,fill=black,inner sep=0pt,minimum size=3pt,label=below:{$\alpha_7$}] {}
-- (6,-1) node[circle,fill=black, inner sep=0pt,minimum size=3pt,label=below:{$\alpha_8$}] {};

\draw
 (2,-0.3) node[circle,inner sep=0pt,minimum size=3pt,label=above:{$\alpha_2$}] {$\times$}
--   (2,-1) node {};
\end{tikzpicture} &
$\begin{matrix} \sigma_1  + k_2\omega_2 &(k_2 \geq 1),\\
\sigma_2  + k_2\omega_2 &(k_2 \geq 2)
\\
\\
\\
\end{matrix}$\\
\hline

(vi)' & 
\begin{tikzpicture}
\draw
    (0,-1) node[circle,inner sep=0pt,minimum size=3pt,label=below:{$\alpha_1$}] {$\Delta$}
 -- (1,-1) node[circle,,inner sep=0pt,minimum size=3pt,label=below:{$\alpha_3$}] {$\Delta$}
--   (2,-1) node[circle,inner sep=0pt,minimum size=3pt,label=below:{$\alpha_4$}] {$\Delta$}
 -- (3,-1) node[circle,inner sep=0pt,minimum size=3pt,label=below:{$\alpha_5$}] {$\times$}
-- (4,-1) node[circle,fill=black,inner sep=0pt,minimum size=3pt,label=below:{$\alpha_6$}] {}
-- (5,-1) node[circle,fill=black,inner sep=0pt,minimum size=3pt,label=below:{$\alpha_7$}] {}
-- (6,-1) node[circle,fill=black, inner sep=0pt,minimum size=3pt,label=below:{$\alpha_8$}] {};

\draw
 (2,-0.3) node[circle,inner sep=0pt,minimum size=3pt,label=above:{$\alpha_2$}] {$\Delta$}
--   (2,-1) node {};
\end{tikzpicture} &
$\begin{matrix} \sigma_1  + k_5\omega_5 &(k_5 \geq 1),\\
\sigma_2  + k_5\omega_5 &(k_5 \geq 2)
\\
\\
\\
\end{matrix}$\\
\hline

(vii) & 
\begin{tikzpicture}
\draw
    (0,-1) node[circle,fill=black,inner sep=0pt,minimum size=3pt,label=below:{$\alpha_1$}] {}
 -- (1,-1) node[circle,,inner sep=0pt,minimum size=3pt,label=below:{$\alpha_3$}] {$\Delta$}
--   (2,-1) node[circle,inner sep=0pt,minimum size=3pt,label=below:{$\alpha_4$}] {$\Delta$}
 -- (3,-1) node[circle,inner sep=0pt,minimum size=3pt,label=below:{$\alpha_5$}] {$\Delta$}
-- (4,-1) node[circle,inner sep=0pt,minimum size=3pt,label=below:{$\alpha_6$}] {$\Delta$}
-- (5,-1) node[circle,fill=black,inner sep=0pt,minimum size=3pt,label=below:{$\alpha_7$}] {}
-- (6,-1) node[circle,fill=black, inner sep=0pt,minimum size=3pt,label=below:{$\alpha_8$}] {};

\draw
 (2,-0.3) node[circle,inner sep=0pt,minimum size=3pt,label=above:{$\alpha_2$}] {$\times$}
--   (2,-1) node {};
\end{tikzpicture} &
$\begin{matrix} \sigma_1  + k_2\omega_2 &(k_2 \geq 1),\\
\sigma_2  + k_2\omega_2 &(k_2 \geq 2)
\\
\\
\\
\end{matrix}$\\
\hline

(vii)' & 
\begin{tikzpicture}
\draw
    (0,-1) node[circle,fill=black,inner sep=0pt,minimum size=3pt,label=below:{$\alpha_1$}] {}
 -- (1,-1) node[circle,,inner sep=0pt,minimum size=3pt,label=below:{$\alpha_3$}] {$\times$}
--   (2,-1) node[circle,inner sep=0pt,minimum size=3pt,label=below:{$\alpha_4$}] {$\Delta$}
 -- (3,-1) node[circle,inner sep=0pt,minimum size=3pt,label=below:{$\alpha_5$}] {$\Delta$}
-- (4,-1) node[circle,inner sep=0pt,minimum size=3pt,label=below:{$\alpha_6$}] {$\Delta$}
-- (5,-1) node[circle,fill=black,inner sep=0pt,minimum size=3pt,label=below:{$\alpha_7$}] {}
-- (6,-1) node[circle,fill=black, inner sep=0pt,minimum size=3pt,label=below:{$\alpha_8$}] {};

\draw
 (2,-0.3) node[circle,inner sep=0pt,minimum size=3pt,label=above:{$\alpha_2$}] {$\Delta$}
--   (2,-1) node {};
\end{tikzpicture} &
$\begin{matrix} \sigma_1  + k_3\omega_3 &(k_3 \geq 1),\\
\sigma_2  + k_3\omega_3 &(k_3 \geq 2)
\\
\\
\\
\end{matrix}$\\
\hline

(viii) & 
\begin{tikzpicture}
\draw
    (0,-1) node[circle,fill=black,inner sep=0pt,minimum size=3pt,label=below:{$\alpha_1$}] {}
 -- (1,-1) node[circle,,inner sep=0pt,minimum size=3pt,label=below:{$\alpha_3$}] {$\Delta$}
--   (2,-1) node[circle,inner sep=0pt,minimum size=3pt,label=below:{$\alpha_4$}] {$\Delta$}
 -- (3,-1) node[circle,inner sep=0pt,minimum size=3pt,label=below:{$\alpha_5$}] {$\times$}
-- (4,-1) node[circle,fill=black,inner sep=0pt,minimum size=3pt,label=below:{$\alpha_6$}] {}
-- (5,-1) node[circle,fill=black,inner sep=0pt,minimum size=3pt,label=below:{$\alpha_7$}] {}
-- (6,-1) node[circle,fill=black, inner sep=0pt,minimum size=3pt,label=below:{$\alpha_8$}] {};

\draw
 (2,-0.3) node[circle,inner sep=0pt,minimum size=3pt,label=above:{$\alpha_2$}] {$\Delta$}
--   (2,-1) node {};
\end{tikzpicture} &
$\begin{matrix} \sigma_1  + k_5\omega_2 &(k_5 \geq 1),\\
\sigma_2  + k_5\omega_2 &(k_5 \geq 2)
\\
\\
\\
\end{matrix}$\\
\hline

(viii)' & 
\begin{tikzpicture}
\draw
    (0,-1) node[circle,fill=black,inner sep=0pt,minimum size=3pt,label=below:{$\alpha_1$}] {}
 -- (1,-1) node[circle,,inner sep=0pt,minimum size=3pt,label=below:{$\alpha_3$}] {$\Delta$}
--   (2,-1) node[circle,inner sep=0pt,minimum size=3pt,label=below:{$\alpha_4$}] {$\Delta$}
 -- (3,-1) node[circle,inner sep=0pt,minimum size=3pt,label=below:{$\alpha_5$}] {$\Delta$}
-- (4,-1) node[circle,fill=black,inner sep=0pt,minimum size=3pt,label=below:{$\alpha_6$}] {}
-- (5,-1) node[circle,fill=black,inner sep=0pt,minimum size=3pt,label=below:{$\alpha_7$}] {}
-- (6,-1) node[circle,fill=black, inner sep=0pt,minimum size=3pt,label=below:{$\alpha_8$}] {};

\draw
 (2,-0.3) node[circle,inner sep=0pt,minimum size=3pt,label=above:{$\alpha_2$}] {$\times$}
--   (2,-1) node {};
\end{tikzpicture} &
$\begin{matrix}\sigma_1  + k_2\omega_2 &(k_2 \geq 1),\\
\sigma_2  + k_2\omega_2 &(k_2 \geq 2)
\\
\\
\\
\end{matrix}$\\
\hline

(viii)'' & 
\begin{tikzpicture}
\draw
    (0,-1) node[circle,fill=black,inner sep=0pt,minimum size=3pt,label=below:{$\alpha_1$}] {}
 -- (1,-1) node[circle,,inner sep=0pt,minimum size=3pt,label=below:{$\alpha_3$}] {$\times$}
--   (2,-1) node[circle,inner sep=0pt,minimum size=3pt,label=below:{$\alpha_4$}] {$\Delta$}
 -- (3,-1) node[circle,inner sep=0pt,minimum size=3pt,label=below:{$\alpha_5$}] {$\Delta$}
-- (4,-1) node[circle,fill=black,inner sep=0pt,minimum size=3pt,label=below:{$\alpha_6$}] {}
-- (5,-1) node[circle,fill=black,inner sep=0pt,minimum size=3pt,label=below:{$\alpha_7$}] {}
-- (6,-1) node[circle,fill=black, inner sep=0pt,minimum size=3pt,label=below:{$\alpha_8$}] {};

\draw
 (2,-0.3) node[circle,inner sep=0pt,minimum size=3pt,label=above:{$\alpha_2$}] {$\Delta$}
--   (2,-1) node {};
\end{tikzpicture} &
$\begin{matrix} \sigma_1  + k_3\omega_3 &(k_3 \geq 1),\\
\sigma_2 + k_3\omega_3 &(k_3 \geq 2)
\\
\\
\\
\end{matrix}$\\
\hline

\end{longtable}
\end{center}
In all cases in the table, when a $K-$type in the right column is not bottom layer, then (by linear algebra) the $2\lambda_L-$value of the infinitesimal character of the corresponding representation is not regular
  integral. Therefore, these representations cannot appear in $\widehat{G}^d$ by \eqref{eq-HP}.   
  For the rest of this section, we illustrate some examples of checking non-regularity or non-integrality of $2\lambda$.

\medskip
Note that if $J_{\mathcal{F}_{\ell}}(\nu_{\ell},\nu_{\ell})$ has indefinite form in the adjoint representation
$V_{\frf_{\ell} \cap \frk}(\sigma_1)$, then $\sigma_1 + \eta$ is always $M-$bottom layer
for all lowest $K-$types $\eta$. Therefore, we are left to study spherical representations $J_{\mathcal{F}_{\ell}}(\nu_{\ell},\nu_{\ell})$ of Type $A_p$ 
($3 \leq p \leq 7$) such that $2\nu_{\ell}$ is regular integral, and
the Hermitian form on $J_{\mathcal{F}_{\ell}}(\nu_{\ell},\nu_{\ell})$ is positive definite on the adjoint representation. 

By the classification of the unitary dual of $GL(n,\mathbb{C})$ in \cite{V1}, this implies $2\nu_{\ell}$ is obtained by concatenating the following string of integers (in the usual $\mathfrak{sl}(p+1)$ coordinates):
\begin{align*}
&(7,5,3,1,-1,-3,-5,-7),\quad (6,4,2,0,-2,-4,-6),\quad (5,3,1,-1,-3,-5),\\
& (4,2,0,-2,-4),\quad (3,1,-1,-3),\quad (2,0,-2), \quad (1,-1) ,\quad (0);\end{align*}
corresponding to the trivial representation, and/or
\begin{equation} \label{eq-comp2}
\begin{aligned}
&\left(\frac{7}{2},\frac{3}{2},-\frac{1}{2},-\frac{5}{2} \ ;\   \frac{5}{2}, \frac{1}{2}, -\frac{3}{2}, -\frac{7}{2}\right), \left(\frac{5}{2},\frac{1}{2},-\frac{3}{2} \ ;\   \frac{3}{2},-\frac{1}{2},-\frac{5}{2}\right), 
\left(\frac{3}{2},-\frac{1}{2} \ ;\   \frac{1}{2},-\frac{3}{2}\right), \left(\frac{1}{2} \ ;\   -\frac{1}{2}\right)
\end{aligned}
\end{equation}
corresponding to the mid-point of Stein's complementary series (whose Dirac cohomology is zero), and/or the following strings:
\begin{equation} \label{eq-comp}
\begin{aligned}
&m=4:\ \left(\frac{9}{2},\frac{5}{2},\frac{1}{2},-\frac{3}{2} \ ;\   \frac{3}{2}, -\frac{1}{2}, -\frac{5}{2}, -\frac{9}{2}\right),\\
&m=3:\ \left(\frac{11}{2},\frac{7}{2},\frac{3}{2},-\frac{1}{2} \ ;\   \frac{1}{2}, -\frac{3}{2}, -\frac{7}{2}, -\frac{11}{2}\right), \left(\frac{7}{2},\frac{1}{2},-\frac{3}{2} \ ;\   \frac{3}{2},-\frac{1}{2},-\frac{5}{2}\right)\\
&m=2:\ \left(\frac{13}{2},\frac{9}{2},\frac{5}{2},\frac{1}{2} \ ;\   -\frac{1}{2}, -\frac{5}{2}, -\frac{9}{2},-\frac{1}{2}\right), \left(\frac{9}{2},\frac{5}{2},\frac{1}{2} \ ;\   -\frac{1}{2},-\frac{5}{2},-\frac{9}{2}\right), \left(\frac{5}{2},\frac{1}{2} \ ;\   -\frac{1}{2},-\frac{5}{2}\right)
\end{aligned}
\end{equation}
corresponding to a non-unitary representation whose first occurrence of indefinite forms is on $V_{\frf_{\ell} \cap \frk}(\sigma_m)$.

\medskip
The requirement that $2\nu_{\ell}$ is regular integral implies that
when $p$ is even, the only choice of $2\nu_{\ell}$ is $(p, p-2, \dots, -p+2, -p)$, i.e. $J_{\mathcal{F}_{\ell}}(\nu_{\ell},\nu_{\ell})$ is the trivial representation. When $p$ is odd,  the regularity 
condition rules out all but  the trivial representation and
parameters \eqref{eq-comp2} and \eqref{eq-comp}.  

\begin{example} \label{eg-a5}
\begin{itemize}
\item[(a)] In Case (i), $M$ consists of a single simple factor $\mathcal{F}_{\ell}$ of Type $A_7$. By the last column of the table and the above discussions, one only needs to consider the case of $J_{A_7}(\nu,\nu)$ with
$$2\nu = \left(\frac{7}{2},\frac{3}{2},-\frac{1}{2},-\frac{5}{2} \ ;\   \frac{5}{2}, \frac{1}{2}, -\frac{3}{2}, -\frac{7}{2}\right) \sim [1,1,1,1,1,1,1]$$
which gives a unitary module for $\eta = k_2\omega_2$ for all $k_2 \geq 1$, or
\begin{equation*}
2\nu = \begin{cases} 
\left(\frac{9}{2},\frac{5}{2},\frac{1}{2},-\frac{3}{2} \ ;\   \frac{3}{2}, -\frac{1}{2}, -\frac{5}{2}, -\frac{9}{2}\right) \sim [2,2,2,1,2,2,2] &\text{and}\quad \eta = \omega_2\\ 
\left(\frac{11}{2},\frac{7}{2},\frac{3}{2},-\frac{1}{2} \ ;\   \frac{1}{2}, -\frac{3}{2}, -\frac{7}{2}, -\frac{11}{2}\right)\sim [2,2,1,1,1,2,2] &\text{and}\quad \eta = \omega_2, 2\omega_2\\ 
\left(\frac{13}{2},\frac{9}{2},\frac{5}{2},\frac{1}{2} \ ;\   -\frac{1}{2}, -\frac{5}{2}, -\frac{9}{2},-\frac{1}{2}\right) \sim [2,1,1,1,1,1,2] &\text{and}\quad \eta = \omega_2, 2\omega_2, 3\omega_2
\end{cases}
\end{equation*}
where bottom layer $K$-type argument fails to detect non-unitarity. 

Consider the induced representation in Step (i) of Section \ref{sec-strategy}:
\begin{equation*} \label{eq-a7}
\mathrm{Ind}_{MN}^{E_8}\left(J_{A_7}(\nu,\nu) \otimes \mathbb{C}_{\tau(\eta)} \otimes {\bf 1} \right)
\end{equation*}
The infinitesimal character of the above induced module (and $\pi$) is equal to $(\la_L,\la_R)$, where
$$2\la_L = 2\widetilde{\nu} + \eta$$
where $2\widetilde{\nu} = \begin{cases} 
[1,x,1,1,1,1,1,1]\\
[2,x,2,2,1,2,2,2]\\ 
[2,x,2,1,1,1,2,2]\\ 
[2,x,1,1,1,1,1,2]
\end{cases}$ satisfying $\langle \widetilde{\nu}, \eta \rangle = 0.$
By solving $x$, one has:
$$2\la_L  = \begin{cases}
\left[1,k_2-\frac{15}{2},1,1,1,1,1,1\right] & k_2 \geq 1\\
\left[2,k_2-\frac{27}{2},2,2,1,2,2,2\right] & k_2 = 1\\
\left[2,k_2-\frac{21}{2},2,1,1,1,2,2\right] & k_2 = 1,2\\
\left[2,k_2-\frac{17}{2},1,1,1,1,1,2\right] & k_2 = 1,2,3
\end{cases}
$$
for $\eta = k_2\omega_2$. Note that none satisfies \eqref{eq-HP}. Consequently, the only possible $J_{A_7}(\nu,\nu) =\mathrm{triv}_{A_7}$ is the trivial representation.

\medskip

\item[(b)] Consider Case (v), where $M$ has a factor of Type $A_5$ and $\eta = k_1\omega_1 + k_2\omega_2 + k_8\omega_8$.
As discussed above, we only need to study $J_{A_5}(\nu,\nu)$ with:
\begin{equation*}
2\nu = \begin{cases}
\left(\frac{5}{2},\frac{1}{2},-\frac{3}{2} \ ;\   \frac{3}{2},-\frac{1}{2},-\frac{5}{2}\right) \sim [1,1,1,1,1] &\text{and}\quad \eta = k_1\omega_1+k_2\omega_2+k_8\omega_8\\
\left(\frac{7}{2},\frac{3}{2},-\frac{1}{2} \ ;\   \frac{1}{2},-\frac{3}{2},-\frac{7}{2}\right) \sim [2,1,1,1,2] &\text{and}\quad \eta = \omega_1+k_2\omega_2+k_8\omega_8\\
\left(\frac{9}{2},\frac{5}{2},\frac{1}{2} \ ;\   -\frac{1}{2},-\frac{5}{2},-\frac{9}{2}\right) \sim [2,2,1,2,2] &\text{and}\quad \eta = \omega_1+k_2\omega_2+k_8\omega_8
\end{cases}
\end{equation*}
More explicitly, the first module is unitary, and the non-unitarity of the two modules are detected by $V_{\frs\fru(6)}(\sigma_2)$ and $V_{\frs\fru(6)}(\sigma_3)$ respectively. 

\medskip
By the same arguments as in (a), let $(\lambda_L,\lambda_R)$ be the infinitesimal character  of the induced module
$$\mathrm{Ind}_{MN}^{E_8}\left( J_{A_5}(\nu;\nu) \otimes \mathbb{C}_{\tau(\eta)} \otimes {\bf 1} \right)$$ along with its lowest $K-$type subquotient $\pi$. Then  $2\la_L$ is equal to:
$$2\la_L  = 2\widetilde{\nu} + \eta = \begin{cases}
\left[k_1-\frac{5}{2},k_2-4,1,1,1,1,1,k_8-\frac{5}{2}\right]\\
\left[k_1-\frac{7}{2},-4,2,1,1,1,2,k_8-\frac{7}{2}\right]\\
\left[k_1-\frac{9}{2},-6,2,2,1,2,2,k_8-\frac{9}{2}\right] 
\end{cases},$$
none of which is integral. Therefore, all of them must have zero Dirac cohomology, and the only possible $J_{A_5}(\nu;\nu)  =\mathrm{triv}_{A_5}$ having nonzero Dirac cohomology is the trivial module. \end{itemize} \end{example}

In general, the Levi subgroup may have more than one simple factor of Type $A$. For instance, consider $I(\widetilde{M}) = \{1,2,4,5,6,7\}$  and 
$$\widetilde{M} = \mathcal{F} \times \mathcal{F}' \times (\mathbb{C}^*)^2,$$ 
where $\mathcal{F}$ is of Type $A_1$ corresponding to $\alpha_1$, and $\mathcal{F}'$ is of Type $A_5$ corresponding to the nodes $\alpha_2, \alpha_4, \alpha_5,\alpha_6, \alpha_7$. Consider the subquotient of the induced module
\begin{equation*}
\mathrm{Ind}_{\widetilde{M}\widetilde{N}}^{E_8}\left(J_{A_1}(\nu,\nu) \otimes J_{A_5}(\nu',\nu') \otimes \mathbb{C}_{\widetilde{\tau}(\eta)} \otimes {\bf 1} \right),
\end{equation*}
where $J_{A_1}(\nu,\nu)$, $J_{A_5}(\nu',\nu')$ with $\eta = k_3\omega_3+k_8\omega_8$. By arguments similar to that of Example \ref{eg-a5}(b), the infinitesimal character of the above induced module does not satisfy \eqref{eq-HP}
for all possible $k_3, k_8 \geq 1$ unless both of them are trivial modules.

\bigskip
In conclusion, if $M$ consists only of Type $A$ factors, the spherical factors $J_A(\nu,\nu)$ must be trivial modules. Otherwise, the parameter $2\lambda$ in $\pi = J(\lambda,-s\lambda)$ would fail to be integral, violating Equation \eqref{eq-HP}. This completes the proof of Proposition \ref{cor-a}.\end{proof}

We end this section by the following remark:
\begin{remark} \label{rmk-other}
The above arguments hold for any Levi subgroup $M$ with  a Type $A$ factor $\mathcal{F}_{\ell}$ such
that when any $j \notin I(M)$ is connected to a node of 
$\mathcal{F}_{\ell}$, it is connected to an endpoint of the diagram. For instance, if $I(M) = \{2,3,4,5,7,8\}$, then $M$ has two
simple factors, one of Type $D_4$ and another of Type $A_2$. Then the module on Type $A_2$ factor is trivial.  
\end{remark}

\section{Type $D$ Levi subgroups}
We now deal with the case when $M$ contains a Type $D$ factor
$\mathcal{F}$ of rank 4,5,6 and 7. Fix the simple roots of $\mathcal{F}$ by:
\begin{center}
\begin{tikzpicture}
\draw
    (-1,0) node[circle,fill=black,inner sep=0pt,minimum size=3pt,label=below:{$\beta_1$}] {}
 -- (0,0) node[circle,fill=black,inner sep=0pt,minimum size=3pt,label=below:{$\beta_2$}] {};
 
\draw (0.5,0) node {$\cdots$};

\draw
    (1,0) node[circle,fill=black,inner sep=0pt,minimum size=3pt,label=below:{$\beta_{p-3}$}] {}
 -- (2,0) node {};

\draw
 (2,0.7) node[circle,fill=black,inner sep=0pt,minimum size=3pt,label=above:{$\beta_{p-1}$}] {}
--  (2,0) node[circle,fill=black,inner sep=0pt,minimum size=3pt,label=below:{$\beta_{p-2}$}] {}
 -- (3,0) node[circle,fill=black,inner sep=0pt,minimum size=3pt,label=below:{$\beta_p$}] {};
\end{tikzpicture}
\end{center}

Here is the main result of this section:
\begin{proposition} \label{cor-d}
Let $\pi = J(\lambda,-s\lambda) \in \widehat{G}$ be such that $2\lambda$ is regular integral (e.g. $\pi \in \widehat{G}^d$), and 
the highest weight of its lowest $K-$type $\eta = \{\lambda + s\lambda\}$ defines a Levi subgroup 
$M$ (c.f. \eqref{eq-m}) consisting only of a Type $D$ factor and possibly a Type $A$ factor. Then $\pi$ must be the lowest $K-$type subquotient of the unitarily induced module
$$\mathrm{Ind}_{M'N'}^{E_8}\left(\pi_D^{unip,d} \otimes \mathrm{triv}_A \otimes \mathbb{C}_{\tau'(\eta)} \otimes {\bf 1}\right),$$
where $M'$ is a Levi subgroup consisting of a Type $D$ simple factor and possibly a Type $A$ simple factor, $\pi_D^{unip,d}$
is a unipotent representation of the Type $D$ factor with nonzero Dirac cohomology, and $\mathrm{triv}_A$ is the trivial representation
of the Type $A$ factor.
\end{proposition}
\begin{proof}
By Remark \ref{rmk-other}, if there is another simple factor in $M$ (which is necessarily of Type $A$), the spherical representation 
corresponding to this Type $A$ factor must be the trivial module. Therefore, one only needs to focus  on
the spherical representation corresponding to $\mathcal{F}$, and the proof follows the pattern of the previous section. 

By \cite[Section 6]{BDW}, the non-unitarity certificates of $(\frf, \mathcal{F} \cap K)-$modules with half-integral
regular infinitesimal character have highest weights:
\begin{itemize}
\item Type $D_7$:
\begin{align*}
&\sigma^7_0 := 2\beta_1 + 2\beta_2 + 2\beta_3 + 2\beta_4 + 2\beta_5 + \beta_6 + \beta_7, \\
&\sigma^7_1 := \beta_1 + 2\beta_2 + 2\beta_3 + 2\beta_4 + 2\beta_5 + \beta_6 + \beta_7, \\
&\sigma^7_2 := \beta_1 + 2\beta_2 + 3\beta_3 + 4\beta_4 + 4\beta_5 + 2\beta_6 + 2\beta_7,\\
&\sigma^7_3 := \beta_1 + 2\beta_2 + 3\beta_3 + 4\beta_4 + 5\beta_5 + 3\beta_6 + 3\beta_7
\end{align*}
\item Type $D_6$:
\begin{align*}
&\sigma^6_0 := 2\beta_1 + 2\beta_2 + 2\beta_3 + 2\beta_4 + \beta_5 + \beta_6, \\
&\sigma^6_1 := \beta_1 + 2\beta_2 + 2\beta_3 + 2\beta_4 + \beta_5 + \beta_6, \\
&\sigma^6_2 := \beta_1 + 2\beta_2 + 3\beta_3 + 4\beta_4 + 2\beta_5 + 2\beta_6,\\
&\sigma^6_3 := \beta_1 + 2\beta_2 + 3\beta_3 + 4\beta_4 + 3\beta_5 + 2\beta_6 
\end{align*}
\item Type $D_5$:
\begin{align*}
&\sigma^5_0 := 2\beta_1 + 2\beta_2 + 2\beta_3 + \beta_4 + \beta_5, \\
&\sigma^5_1 := \beta_1 + 2\beta_2 + 2\beta_3 + \beta_4 + \beta_5, \\
&\sigma^5_2 := \beta_1 + 2\beta_2 + 3\beta_3 + 2\beta_4 + 2\beta_5
\end{align*}
\item Type $D_4$:
\begin{align*}
&\sigma^4_0 := 2\beta_1 + 2\beta_2 + \beta_3 + \beta_4, \\
&\sigma^4_1 := \beta_1 + 2\beta_2 + \beta_3 + \beta_4, \\
&\sigma^4_2 := \beta_1 + 2\beta_2 + 2\beta_3 + \beta_4
\end{align*}
\end{itemize}

As in the previous section, we list the possibilities of the Type $D_p$ factor and the possibilities
of $\eta$ so that $\sigma + \eta$ is bottom layer:
\begin{center}
\begin{longtable}{|c|c|c|c|}
\hline
Type & $I(F)$ (in $\Delta$) and $j$ (in $\times$) &  Bottom layer $\sigma + \eta$\\
\hline \hline
$D_7$ & 
\begin{tikzpicture}
\draw
    (0,-1) node[circle,inner sep=0pt,minimum size=3pt,label=below:{$\alpha_1$}] {$\times$}
 -- (1,-1) node[circle,inner sep=0pt,minimum size=3pt,label=below:{$\alpha_3$}] {$\Delta$}
--   (2,-1) node[circle,inner sep=0pt,minimum size=3pt,label=below:{$\alpha_4$}] {$\Delta$}
 -- (3,-1) node[circle,inner sep=0pt,minimum size=3pt,label=below:{$\alpha_5$}] {$\Delta$}
-- (4,-1) node[circle,inner sep=0pt,minimum size=3pt,label=below:{$\alpha_6$}] {$\Delta$}
-- (5,-1) node[circle,inner sep=0pt,minimum size=3pt,label=below:{$\alpha_7$}] {$\Delta$}
-- (6,-1) node[circle,inner sep=0pt,minimum size=3pt,label=below:{$\alpha_8$}] {$\Delta$};

\draw
 (2,-0.3) node[circle,inner sep=0pt,minimum size=3pt,label=above:{$\alpha_2$}] {$\Delta$}
--   (2,-1) node {};
\end{tikzpicture} &
$\begin{matrix} \sigma^7_0 + k_1\omega_1 &(k_1 \geq 1),\\
\sigma^7_1 + k_1\omega_1 &(k_1 \geq 1),\\
\sigma^7_2 + k_1\omega_1 &(k_1 \geq 2),\\
\sigma^7_3 + k_1\omega_1 &(k_1 \geq 3)
\\
\\
\\
\end{matrix}$
\\

\hline
$D_6$ & 
\begin{tikzpicture}
\draw
    (0,-1) node[circle,inner sep=0pt,minimum size=3pt,label=below:{$\alpha_1$}] {$\times$}
 -- (1,-1) node[circle,inner sep=0pt,minimum size=3pt,label=below:{$\alpha_3$}] {$\Delta$}
--   (2,-1) node[circle,inner sep=0pt,minimum size=3pt,label=below:{$\alpha_4$}] {$\Delta$}
 -- (3,-1) node[circle,inner sep=0pt,minimum size=3pt,label=below:{$\alpha_5$}] {$\Delta$}
-- (4,-1) node[circle,inner sep=0pt,minimum size=3pt,label=below:{$\alpha_6$}] {$\Delta$}
-- (5,-1) node[circle,inner sep=0pt,minimum size=3pt,label=below:{$\alpha_7$}] {$\Delta$}
-- (6,-1) node[circle,inner sep=0pt,minimum size=3pt,label=below:{$\alpha_8$}] {$\times$};

\draw
 (2,-0.3) node[circle,inner sep=0pt,minimum size=3pt,label=above:{$\alpha_2$}] {$\Delta$}
--   (2,-1) node {};
\end{tikzpicture} &
$\begin{matrix} \sigma^6_0  + k_1\omega_1 + k_8\omega_8 &(k_1 \geq 1, k_8 \geq 2),\\
\sigma^6_1  + k_1\omega_1 + k_8\omega_8 &(k_1, k_8 \geq 1),\\
\sigma^6_2  + k_1\omega_1 + k_8\omega_8 &(k_1 \geq 2, k_8 \geq 1)\\
\sigma^6_3  + k_1\omega_1 + k_8\omega_8 &(k_1 \geq 3, k_8 \geq 1)
\\
\\
\\
\end{matrix}$\\
\hline

$D_5$ & 
\begin{tikzpicture}
\draw
    (0,-1) node[circle,inner sep=0pt,minimum size=3pt,label=below:{$\alpha_1$}] {$\times$}
 -- (1,-1) node[circle,inner sep=0pt,minimum size=3pt,label=below:{$\alpha_3$}] {$\Delta$}
--   (2,-1) node[circle,inner sep=0pt,minimum size=3pt,label=below:{$\alpha_4$}] {$\Delta$}
 -- (3,-1) node[circle,inner sep=0pt,minimum size=3pt,label=below:{$\alpha_5$}] {$\Delta$}
-- (4,-1) node[circle,inner sep=0pt,minimum size=3pt,label=below:{$\alpha_6$}] {$\Delta$}
-- (5,-1) node[circle,inner sep=0pt,minimum size=3pt,label=below:{$\alpha_7$}] {$\times$}
-- (6,-1) node[circle,fill=black,inner sep=0pt,minimum size=3pt,label=below:{$\alpha_8$}] {};

\draw
 (2,-0.3) node[circle,inner sep=0pt,minimum size=3pt,label=above:{$\alpha_2$}] {$\Delta$}
--   (2,-1) node {};
\end{tikzpicture} &
$\begin{matrix} \sigma^5_0  + k_1\omega_1 + k_7\omega_7 &(k_1 \geq 1, k_7 \geq 2),\\
\sigma^5_1  + k_1\omega_1 + k_7\omega_7 &(k_1, k_7 \geq 1),\\
\sigma^5_2  + k_1\omega_1 + k_7\omega_7 &(k_1 \geq 2, k_7 \geq 1)\\
\\
\\
\\
\end{matrix}$\\
\hline

$D_5'$ & 
\begin{tikzpicture}
\draw
    (0,-1) node[circle,inner sep=0pt,minimum size=3pt,label=below:{$\alpha_1$}] {$\Delta$}
 -- (1,-1) node[circle,inner sep=0pt,minimum size=3pt,label=below:{$\alpha_3$}] {$\Delta$}
--   (2,-1) node[circle,inner sep=0pt,minimum size=3pt,label=below:{$\alpha_4$}] {$\Delta$}
 -- (3,-1) node[circle,inner sep=0pt,minimum size=3pt,label=below:{$\alpha_5$}] {$\Delta$}
-- (4,-1) node[circle,inner sep=0pt,minimum size=3pt,label=below:{$\alpha_6$}] {$\times$}
-- (5,-1) node[circle,fill=black, inner sep=0pt,minimum size=3pt,label=below:{$\alpha_7$}] {}
-- (6,-1) node[circle,fill=black, inner sep=0pt,minimum size=3pt,label=below:{$\alpha_8$}] {};

\draw
 (2,-0.3) node[circle,inner sep=0pt,minimum size=3pt,label=above:{$\alpha_2$}] {$\Delta$}
--   (2,-1) node {};
\end{tikzpicture} &
$\begin{matrix} \sigma^5_0  + k_6\omega_6 &(k_6 \geq 1),\\
\sigma^5_1  + k_6\omega_6 &(k_6 \geq 1),\\
\sigma^5_2  + k_6\omega_6 &(k_6 \geq 2)
\\
\\
\\
\end{matrix}$\\
\hline

$D_4$ & 
\begin{tikzpicture}
\draw
    (0,-1) node[circle,inner sep=0pt,minimum size=3pt,label=below:{$\alpha_1$}] {$\times$}
 -- (1,-1) node[circle,inner sep=0pt,minimum size=3pt,label=below:{$\alpha_3$}] {$\Delta$}
--   (2,-1) node[circle,inner sep=0pt,minimum size=3pt,label=below:{$\alpha_4$}] {$\Delta$}
 -- (3,-1) node[circle,inner sep=0pt,minimum size=3pt,label=below:{$\alpha_5$}] {$\Delta$}
-- (4,-1) node[circle,inner sep=0pt,minimum size=3pt,label=below:{$\alpha_6$}] {$\times$}
-- (5,-1) node[circle,fill=black, inner sep=0pt,minimum size=3pt,label=below:{$\alpha_7$}] {}
-- (6,-1) node[circle,fill=black, inner sep=0pt,minimum size=3pt,label=below:{$\alpha_8$}] {};

\draw
 (2,-0.3) node[circle,inner sep=0pt,minimum size=3pt,label=above:{$\alpha_2$}] {$\Delta$}
--   (2,-1) node {};
\end{tikzpicture} &
$\begin{matrix} \sigma^4_0  + k_1\omega_1 + k_6\omega_6 &(k_1 \geq 2, k_6 \geq 1),\\
\sigma^4_1  + k_1\omega_1 + k_6\omega_6 &(k_1, k_6 \geq 1),\\
\sigma^4_2  + k_1\omega_1 + k_6\omega_6 &(k_1 \geq 2, k_6 \geq 1),
\\
\\
\\
\end{matrix}$\\
\hline

\end{longtable}
\end{center}

As in the previous section, we are reduced to the cases when the non-unitarity of 
the spherical module $J_{D_p}(\nu,\nu)$ is detected by a $\sigma_i^p$
such that  $\sigma_i^p + \eta$ is not $M-$bottom layer. In most
cases, the infinitesimal character does not satisfy the condition for
the Dirac cohomology to be nonzero by linear algebra. We give details in the various cases.

\subsection{Type $D_7$} By the above table, one only considers the spherical modules $J_{D_7}(\nu_7,\nu_7)$ 
such that the first occurrence of opposite signatures at the $Spin(14)-$types with highest weight $\sigma_2^7$ and $\sigma_3^7$
only. By the results of \cite[Section 6]{BDW}, they are
\begin{equation}
2\nu_7 = \begin{cases}
[2,2,2,1,1,1,1] &\text{and}\ \ \eta = \omega_1,\\
[2,1,1,1,1,1,1] &\text{and}\ \ \eta = \omega_1, 2\omega_1
\end{cases}
\end{equation} 
In all the above cases, the infinitesimal characters $(\la_L,\la_R)$ of the induced modules
\begin{equation} \label{eq-d71}
\mathrm{Ind}_{MN}^{E_8}\left(J_{D_7}(\nu_7,\nu_7) \otimes \mathbb{C}_{\tau(\eta)} \otimes {\bf 1} \right)
\end{equation}
have $2\la_L$-value equal to:
$$
2\la_L = \begin{cases}
[-\frac{25}{2},1,1,1,1,2,2,2] & \text{if} \ \ 2\nu_{\ell} = [2,2,2,1,1,1,1]\ \text{and}\ \eta = \omega_1,\\
[-10,1,1,1,1,1,1,2] \sim [0,0,1,0,1,0,1,0] & \text{if} \ \ 2\nu_{\ell} = [2,1,1,1,1,1,1]\ \text{and}\ \eta = \omega_1,\\
[-9,1,1,1,1,1,2] \sim [0,1,1,0,0,1,0,1]  & \text{if} \ \ 2\nu_{\ell} = [2,1,1,1,1,1,1]\ \text{and}\ \eta = 2\omega_1,
\end{cases}$$
none of which is regular integral.

\begin{remark} \label{rmk-d7}
If $2\nu_7 = [1,1,1,1,1,1,1]$, $[2,2,1,1,1,1,1]$, $[2,2,2,2,1,1,1]$ ($= (6,5,4,3,2,1,0)$, $(8,6,4,3,2,1,0)$
and $(10,8,6,4,2,1,0)$ in usual coordinates of $D_7$), then $J_{D_7}(\nu_7,\nu_7)$
are spherical unipotent representations obtained by theta-lift of the spherical metaplectic representations
of $Sp(6,\mathbb{C})$, $Sp(4,\mathbb{C})$ and $Sp(2,\mathbb{C})$ to $SO(14,\mathbb{C})$ respectively. 
By \cite{DW1}, only the middle parameter is in the Dirac series. 

On the
other hand, the infinitesimal characters $(\la_L,\la_R)$ of \eqref{eq-d71} corresponding to these three spherical
unipotent representations have $2\la_L-$value equal to:
$$[-\frac{21}{2}+k_1,1,1,1,1,1,1,1],\quad [-12+k_1,1,1,1,1,1,2,2],\quad [-\frac{31}{2}+k_1,1,1,1,2,2,2,2]$$
respectively. Coincidentally, only the middle parameter can possibly satisfy \eqref{eq-HP}. This observation also holds
for the other spherical (as well as nonspherical) unipotent representations of Type $D$ and $E$ which we are going to study 
below (see the paragraphs after \eqref{eq-unipd7}, \eqref{eq-indd6}, \eqref{eq-indd4}, \eqref{eq-inde7}
and Remark \ref{rmk-e7}). 
\end{remark}

\subsection{Type $D_6$} We carry out the same analysis as in the previous section. Suppose
$J_{D_6}(\nu_6,\nu_6)$ is such that the signatures of the Hermitian form are {\bf both} negative at the $Spin(12)-$types 
$V_{\mathfrak{so}(12)}(\sigma_6^0)$ and $V_{\mathfrak{so}(12)}(\sigma_6^2)$ (or 
$V_{\mathfrak{so}(12)}(\sigma_6^0)$ and $V_{\mathfrak{so}(12)}(\sigma_6^3)$), then 
for any $\eta = k_1\omega_8 + k_8\omega_8$, at least one of $\eta + \sigma_6^0$ and
$\eta + \sigma_6^{2}$ (or $\sigma_6^{3}$) is $M-$bottom layer, which implies $\pi$ is not unitary.

Consequently, we only consider $J_{D_6}(\nu_6,\nu_6)$ such that the signature is indefinite
at {\bf exactly} one of $V_{\mathfrak{so}(12)}(\sigma_6^0)$, $V_{\mathfrak{so}(12)}(\sigma_6^2)$ or $V_{\mathfrak{so}(12)}(\sigma_6^3)$. By the results in \cite{BDW}, we are
reduced to studying the following cases:
\begin{itemize}
\item $\sigma_6^2:$ $2\nu_6 = [2,2,2,2,1,1]$ and $\eta = \omega_1 + k_8\omega_8$;
\item $\sigma_6^3:$ $2\nu_6 = [2,2,1,1,1,1]$ and $\eta = \omega_1 + k_8\omega_8, 2\omega_1 + k_8\omega_8$;
\item $\sigma_6^0:$ $2\nu_6 =$ 
\begin{align*}
&[1,2,2,2,2,2],\ [1,1,2,2,2,2],\ [2,1,1,2,2,2],\ [2,2,1,1,2,2],\ [2,2,2,2,1,1],\\
&[1,1,1,2,2,2],\ [1,2,1,1,2,2],\ [1,2,2,1,1,1],\ [1,1,2,1,1,1],\ [1,1,1,1,2,2]
\end{align*}
and $\eta = k_1\omega_1 + \omega_8$ (c.f. \cite[Section 6.3(c)]{BDW}).
\end{itemize}
In the first two cases, the infinitesimal character $(\la_L,\la_R)$ of the induced module 
$$\mathrm{Ind}_{MN}^{E_8}\left(J_{D_6}(\nu_6,\nu_6) \otimes \mathbb{C}_{\tau(\eta)} \otimes {\bf 1} \right)$$
has $2\lambda_L-$value equal to:
$$
2\la_L = \begin{cases}
[-\frac{25}{2},1,1,2,2,2,2,k_8-8] & \text{if} \ \ 2\nu_{\ell} = [2,2,2,2,1,1]\ \text{and}\ \eta = \omega_1 + k_8\omega_8,\\
[-8,1,1,1,1,2,2,k_8-7] & \text{if} \ \ 2\nu_{\ell} = [2,2,1,1,1,1]\ \text{and}\ \eta = \omega_1 + k_8\omega_8,\\
[-7,1,1,1,1,2,2,k_8-7]  & \text{if} \ \ 2\nu_{\ell} = [2,2,1,1,1,1]\ \text{and}\ \eta = 2\omega_1 + k_8\omega_8.
\end{cases}$$
One can check that none of the above $2\la_L$ is regular integral. More precisely, the first parameter is not integral.
As for the second parameter, $[-8,1,1,1,1,2,2,k_8-7] = (0,1,2,3,5,7,k_8,k_8+2)$ in usual Bourbaki coordinates is singular for all $k_8 > 0$ since
$$\langle (0,1,2,3,5,7,k_8,k_8+2), \frac{1}{2}(1,1,-1,-1,-1,1,-1,1)\rangle = 0.$$
Similarly, the last parameter $[-7,1,1,1,1,2,2,k_8-7] = (0,1,2,3,5,7,k_8,k_8+4)$ is singular for all $k_8 > 0$, since 
$$\langle (0,1,2,3,5,7,k_8,k_8+4), \frac{1}{2}(1,-1,1,-1,1,-1,-1,1)\rangle = 0.$$

We are left to study the last case when the non-unitarity certificate is $V_{\mathfrak{so}(12)}(\sigma_6^0)$. Here we need to apply Step (iv) in Section \ref{sec-strategy} by extending the Levi subgroup
$M$ to $M'$ of Type $D_7$ and considering the induced module
\begin{equation} \label{eq-indd7}
\mathrm{Ind}_{M'N'}^{E_8}\left(J_{D_7}(\nu_6^+;\nu_6^-) \otimes \mathbb{C}_{\tau'(\eta)} \otimes {\bf 1} \right)
\end{equation}
where $\nu_6^{\pm} = \left(\pm\frac{1}{2}, \nu_6\right)$ in the usual coordinates of $D_7$, and $\tau'(\eta)$
is chosen such that its lowest $K-$type is $V_{\frk}(k_1\omega_1 + \omega_8) = V_{\frk}(\eta)$ (note that $\eta = (0,0,0,0,0,0,1,2k_1+1)$ in usual coordinates).

We begin by considering the following two choices of $2\nu_6$: 
\begin{center}
$2\nu_6 = [1,1,1,1,2,2] = (6,5,4,3,2,0)$\quad (or $[2,2,1,1,2,2] = (8,6,4,3,2,0)),$
\end{center}
corresponding to the module $J_{D_7}(\nu_6^+;\nu_6^-)$ in \eqref{eq-indd7} with
\begin{equation} \label{eq-unipd7}
\begin{aligned}
&J_{D_7}\left(\frac{1}{2}(6,5,4,3,2,1,0); \frac{1}{2}(6,5,4,3,2,-1,0)\right) \\ 
(\text{or}\ &J_{D_7}\left(\frac{1}{2}(8,6,4,3,2,1,0); \frac{1}{2}(8,6,4,3,2,-1,0)\right)).
\end{aligned}
\end{equation}
Both of them are (non-spherical) unipotent representations obtained by the theta-lift of the non-spherical metaplectic representations in $Sp(4,\mathbb{C})$ (or $Sp(6,\mathbb{C})$) to $SO(14,\mathbb{C})$.
However, in the second case, $2\la_L = 2\nu_6 + \eta = [1,1,1,1,2,2] + k_1\omega_1+\omega_8$ is not integral for all $k_1 \geq 1$. Therefore, only the first module of \eqref{eq-unipd7} (which is in $\widehat{M'}^d$ by \cite{DW1}) contributes to the Dirac series.


\medskip
For the other choices of $2\nu_6$, the module $J_{D_7}(\nu_6^+;\nu_6^-)$ in \eqref{eq-indd7}
has indefinite form on 
$V_{\mathfrak{so}(14)}(1,0,0,0,0,0,0)$ and $V_{\mathfrak{so}(14)}(2,1,0,0,0,0,0)$ by \cite[Section 6.4]{BDW}.
These $(M'\cap K)-$type is $M'-$bottom layer for all $k_1 \geq 1$, since the weight $(0,0,0,0,0,1,2,2k_1+1)$
is dominant in $E_8$ for all $k_1 \geq 1$. Consequently, the induced module \eqref{eq-indd7} and its lowest $K-$type
subquotient $\pi$ are both nonunitary.

\subsection{Type $D_5$} As in the previous section, we only study $J_{D_5}(\nu_5,\nu_5)$ such that the signatures of the Hermitian form the signature is indefinite at {\bf exactly} one of $V_{\mathfrak{so}(10)}(\sigma_5^0)$ or $V_{\mathfrak{so}(10)}(\sigma_5^2)$:
\begin{itemize}
\item $\sigma_5^2:$ $2\nu_5 = [2,2,2,1,1]$ and $\eta = \omega_1 + k_7\omega_7 + k_8\omega_8$;
\item $\sigma_5^0: 2\nu_5 =$ 
$$[1,2,2,2,2],\ [1,1,2,2,2],\ [2,1,1,2,2],\ [1,1,1,2,2],\ [1,2,2,1,1,1],$$  
and $\eta = k_1\omega_1 + \omega_7 + k_8\omega_8$.
\end{itemize}
The analysis is the same as in the Type $D_6$ case, where the $2\la_L$ parameter is not regular integral,
or the irreducible subquotient is not unitary by bottom layer arguments. The interesting cases are $2\nu_5 = [2,1,1,2,2] = (6,4,3,2,0)$ and $[1,1,1,2,2] = (5,4,3,2,0)$, where the induced modules
\begin{equation} \label{eq-indd6}
\begin{aligned}
&\mathrm{Ind}_{M'N'}^{E_8}\left(J_{D_6}\left(\frac{1}{2}(6,4,3,2,1,0); \frac{1}{2}(6,4,3,2,-1,0)\right) \otimes \mathbb{C}_{\tau'(\eta)}\otimes {\bf 1} \right);\\ 
&\mathrm{Ind}_{M'N'}^{E_8}\left(J_{D_6}\left(\frac{1}{2}(5,4,3,2,1,0); \frac{1}{2}(5,4,3,2,-1,0)\right) \otimes \mathbb{C}_{\tau'(\eta)} \otimes {\bf 1} \right)
\end{aligned}
\end{equation}
(here $M'$ has a single simple factor of Type $D_6$) are both unitary. Indeed, the $D_6-$factors in the above equation are theta lifts
from the non-spherical metaplectic representation of $Sp(4,\mathbb{C})$ and $Sp(6,\mathbb{C})$ to $SO(12,\mathbb{C})$. However, 
in the first case, $2\la_L = \widetilde{\nu_5}+k_1\omega_1 + \omega_7 + k_8\omega_8$ is not integral for all $k_1, k_8 \geq 1$. Consequently, only the second module in \eqref{eq-indd6} (whose inducing module $J_{D_6}\left(\frac{1}{2}(5,4,3,2,1,0); \frac{1}{2}(5,4,3,2,-1,0)\right)$
is in $\widehat{M'}^d$ by \cite{DW1}) contributes to $\widehat{G}^d$.


\subsection{Type $D_5'$} As before, one only considers the spherical module $J_{D_5'}(\nu_5,\nu_5)$ 
such that the first occurrence of opposite signature is at the $Spin(10)-$type with highest weight $\sigma_0^5$. 
By the results of \cite{BDW}, there is only one possibility:
\begin{equation}
2\nu_5 = [2,1,1,1,1] \quad \text{and}\quad \eta = \omega_6+k_7\omega_7+k_8\omega_8.
\end{equation} 
In this case, the infinitesimal character $(\la_L,\la_R)$ of the induced module
$$\mathrm{Ind}_{MN}^{E_8}\left(J_{D_5'}(\nu_5,\nu_5) \otimes \mathbb{C}_{\tau(\eta)} \otimes {\bf 1}\right)$$
has $2\la_L$ equal to a non-integral weight $2\la_L = [2,1,1,1,1,-\frac{9}{2},k_7,k_8]$ 
for all $k_7, k_8 \geq 0$. 


\subsection{Type $D_4$} We study $J_{D_4}(\nu_4,\nu_4)$ such that the signature of the Hermitian form is indefinite at {\bf exactly} one of $V_{\mathfrak{so}(8)}(\sigma_4^0)$ or $V_{\mathfrak{so}(8)}(\sigma_4^2)$:
\begin{itemize}
\item $\sigma_4^2:$ $2\nu_4 = [2,2,1,1]$ and $\eta = \omega_1 + k_6\omega_6 + k_7\omega_7 + k_8\omega_8$;
\item $\sigma_4^0:$ $2\nu_4 = [1,2,2,2], [1,1,2,2]$ and $\eta = k_1\omega_1 + \omega_6 + k_7\omega_7 + k_8\omega_8$.
\end{itemize}
As in the analysis in Type $D_5$ and $D_6$ above, the only possibility of $\pi$ being unitary is when $2\nu_4 = [1,1,2,2] = (4,3,2,0)$, where the induced module
\begin{equation} \label{eq-indd4}
\mathrm{Ind}_{M'N'}^{E_8}\left(J_{D_5}\left(\frac{1}{2}(4,3,2,1,0); \frac{1}{2}(4,3,2,-1,0)\right) \otimes \mathbb{C}_{\tau'(\eta)}\otimes {\bf 1} \right)
\end{equation}
(here $M'$ has a single simple factor of Type $D_5$) is unitary with no Dirac cohomology: the $D_5-$factor in the above equation is the theta lift
from the non-spherical metaplectic representation of $Sp(4,\mathbb{C})$ to $SO(10,\mathbb{C})$. However, $2\la_L = \widetilde{\nu_4} + k_1\omega_1 + \omega_6 + k_7\omega_7 + k_8\omega_8$ is not integral, so it does not contribute to $\widehat{G}^d$.


\bigskip
This finishes the proof of Proposition \ref{cor-d}.
\end{proof}

\section{Type $E$ Levi subgroups}
We now study the case when there is a Type $E_6$ or $E_7$ factor in the Levi subgroup $M$ of $G$, or $M = G$ is of Type $E_8$. Throughout this section, we fix the simple roots of $E_i$ by $\beta_i = \alpha_i$ ($1 \leq i \leq 8$),
where $\alpha_i$ are the simple roots of $E_8$ given in \eqref{eq-e8}.

\medskip
Analogous to the previous sections, we will prove the following:

\begin{proposition} \label{cor-e}
Let $\pi = J(\lambda,-s\lambda) \in \widehat{G}$ be such that $2\lambda$ is regular integral (e.g. $\pi \in \widehat{G}^d$), and 
the highest weight of its lowest $K-$type $\eta = \{\lambda + s\lambda\}$ defines a Levi subgroup 
$M$ (c.f. \eqref{eq-m}) consisting only of a Type $E$ simple factor and possibility a Type $A$ simple factor. Then $\pi$ must be a subquotient of the unitarily induced module
$$\mathrm{Ind}_{M'N'}^{E_8}\left(\pi_E^{unip,d} \otimes \mathrm{triv}_A \otimes \mathbb{C}_{\tau'(\eta)} \otimes {\bf 1}\right),$$
where $M'$ is a Levi subgroup consisting of a Type $E$ simple factor and possibility a Type $A$ simple factor, $\pi_E^{unip,d}$
is a unipotent representation of Type $E$ with nonzero Dirac cohomology, 
and $\mathrm{triv}_A$ is the trivial representation
of Type $A$.
\end{proposition}

\begin{proof}
For these groups, we establish the non-unitarity certificates for the
spherical modules on the $K-$types appearing in $\texttt{adj}_i \otimes
\texttt{adj}_i$ where $\texttt{adj}_i$ is the adjoint representation in $E_i$ ($i = 6,7,8$) via intertwining operators using \texttt{mathematica}. 
This will rule out all but three parameters in the case of $E_8$, and the remaining three parameters are widely conjectured to be unitary (see Section \ref{subsec-e8} below).

As for $E_6$ and $E_7$, the highest weights of the $K-$types appearing in $\texttt{adj}_i \otimes \texttt{adj}_i$ are:
\begin{itemize}
\item Type $E_6$:
\begin{align*}
\omega_2 &= \beta_1 + 2\beta_2 + 2\beta_3 + 3\beta_4 + 2\beta_5 + \beta_6, \\
\omega_1+\omega_6 &= 2\beta_1 + 2\beta_2 + 3\beta_3 + 4\beta_4 + 3\beta_5 + 2\beta_6,\\
2\omega_2 &= 2\beta_1 + 4\beta_2 + 4\beta_3 + 6\beta_4 + 4\beta_5 + 2\beta_6, \\
\omega_4 &= \beta_1 + 2\beta_2 + 3\beta_3 + 4\beta_4 + 3\beta_5 + 2\beta_6 
\end{align*}
\item Type $E_7$:
\begin{align*}
\omega_1 &= 2\beta_1 + 2\beta_2 + 3\beta_3 + 4\beta_4 + 3\beta_5 + 2\beta_6 + \beta_7, \\
\omega_3 &= 2\beta_1 + 3\beta_2 + 4\beta_3 + 6\beta_4 + 5\beta_5 + 4\beta_6 + 2\beta_7,\\
\omega_6 &= 3\beta_1 + 4\beta_2 + 6\beta_3 + 8\beta_4 + 6\beta_5 + 4\beta_6 + 2\beta_7,\\
2\omega_1 &= 4\beta_1 + 4\beta_2 + 6\beta_3 + 8\beta_4 + 6\beta_5 + 4\beta_6 + 2\beta_7, 
\end{align*}
\end{itemize}

By the same argument as in the previous sections, consider the lowest $K-$type subquotient $\pi$ of
$$\mathrm{Ind}_{MN}^{E_8}\left(J_{E_i}(\nu_i,\nu_i) \otimes \mathbb{C}_{\tau(\eta)} \otimes {\bf 1} \right), \quad \eta = \begin{cases} k\omega_7 + l\omega_8 & \text{if}\ i = 6\\
k\omega_8 & \text{if}\ i = 7\end{cases}.$$
If the spherical module $J_{E_i}(\nu_i,\nu_i)$ has indefinite form on 
$\texttt{adj}_i \otimes \texttt{adj}_i$, then bottom layer arguments 
imply that $\pi$ is also not unitary for $k\geq 2$. In the case $k = 1$, a majority of $\nu_i$'s and $\eta$'s are ruled out since $2\lambda_L-$parameter does not satisfy the necessary conditions \eqref{eq-HP} for Dirac cohomology to be nonzero. We will study the few remaining representations on a case-by-case basis.

Finally, we are left with the representations with positive Hermitian form on the level of $\texttt{adj}_i \otimes \texttt{adj}_i$. We will either identify them as unipotent representations, or we apply \texttt{atlas} (in the Appendix) to conclude that they are not in the Dirac series.

\subsection{Type $E_8$} \label{subsec-e8} We use \texttt{mathematica} to compute
the signatures of the $2^8 = 256$ choices of Hermitian, spherical modules $J_{E_8}(\nu_8,\nu_8)$
 on the $K-$types appearing in $\texttt{adj}_8 \otimes \texttt{adj}_8$. It turns out 
that only $3$ of them have definite Hermitian form:
$$2\nu_8 = [1,1,1,1,1,1,1,1], \quad 2\nu_8 = [1,1,1,1,1,1,2,2], \quad 2\nu_8 = [2,2,2,2,2,2,2,2].$$
Indeed, these $\nu_8$-parameters correspond to spherical unipotent representations attached to the nilpotent orbits $4A_1$ and $3A_1$
(using the Bala-Carter notation) along with the trivial representation
respectively. Although we cannot use \texttt{atlas} to verify their
unitarity, this is  widely believed to be the case.

\smallskip
Assuming the unitarity of these modules, we now check all these modules are in $\widehat{G}^d$.
The first module $J_{E_8}(\frac{1}{2}[1,1,1,1,1,1,1,1],\frac{1}{2}[1,1,1,1,1,1,1,1])$ is the model representation, 
whose $K-$spectrum is given in \cite{AHV} by
$$J_{E_8}(\frac{1}{2}[1,1,1,1,1,1,1,1],\frac{1}{2}[1,1,1,1,1,1,1,1])|_K = \bigoplus_{a,b,c,d,e,f,g,h \geq 0} V_{\mathfrak{k}}([a,b,c,d,e,f,g,h]).$$
By \eqref{eq-HP}, $V_{\mathfrak{k}}(\tau)$ contributes to Dirac cohomology if and only if $\tau = 2\la_L - \rho = 0$.
Note that $V_{\mathfrak{k}}([a,b,c,d,e,f,g,h]) \otimes S_G = V_{\mathfrak{k}}([a,b,c,d,e,f,g,h]) \otimes \left( 2^{\lfloor \frac{\mathrm{rank}(\mathfrak{g})}{2}\rfloor} V_{\mathfrak{k}}(\rho)\right)$ contains $V_{\mathfrak{k}}(0)$ if and only if 
$$[a,b,c,d,e,f,g,h] = [1,1,1,1,1,1,1,1] = \rho.$$ 
So $V_{\mathfrak{k}}([1,1,1,1,1,1,1,1])$ is the only possible $K-$type contributing to its Dirac cohomology.

\medskip
As for $J_{E_8}(\frac{1}{2}[1,1,1,1,1,1,2,2],\frac{1}{2}[1,1,1,1,1,1,2,2])$, the work of \cite{MG} implies that
$$J_{E_8}(\frac{1}{2}[1,1,1,1,1,1,1,1],\frac{1}{2}[1,1,1,1,1,1,1,1])|_K = \bigoplus_{a,b,c,d \geq 0} V_{\mathfrak{k}}([a,0,0,0,0,b,c,d]).$$
As above, the only possible $\widetilde{K}-$type contributing to Dirac cohomology is $V_{\mathfrak{k}}(2\la_L - \rho) = V_{\mathfrak{k}}([0,0,0,0,0,0,1,1])$. One can check (through \texttt{mathematica} for instance) that
the only possibility for $V_{\mathfrak{k}}([a,0,0,0,0,b,c,d]) \otimes S_G$ containing $V_{\mathfrak{k}}([0,0,0,0,0,0,1,1])$ is 
when $$[a,0,0,0,0,b,c,d] = [4,0,0,0,0,4,1,1].$$ 
Hence this representation is also in $\widehat{G}^d$, and $V_{\mathfrak{k}}([4,0,0,0,0,4,1,1])$ is the only $K-$type contributing to
its Dirac cohomology.

\subsection{Type $E_7$} \label{subsec-e7} In this case, there are $2^7 = 128$ choices of Hermitian, spherical modules $J_{E_7}(\nu_7,\nu_7)$. 
One can check by \texttt{mathematica} that $119$ of them have indefinite Hermitian form on the $K-$types in $\texttt{adj}_7 \otimes \texttt{adj}_7$. 
In all such cases $2\la_L = 2\widetilde{\nu_7}+\omega_8$ is either non-integral, or it is singular. So none of them contributes to the Dirac series. 

As for the remaining $9$ parameters having definite Hermitian form on the $K-$types in  $\texttt{adj}_7 \otimes \texttt{adj}_7$, only $3$ of them give regular integral $2\la_L = 2\widetilde{\nu_7} + k\omega_8$ for $k \geq 1$:
$$2\nu_7 = [1,1,1,1,1,1,2],\quad [2,1,1,1,1,1,2], \quad [2,2,2,2,2,2,2].$$
(some interesting singular integral $2\nu_7$ are given in Remark \ref{rmk-e7} below). Note that the first and third parameter above give the spherical unipotent representation corresponding to the orbit $3A_1'$
(the $21^{st}$-entry in \cite[Table 4]{DW3}) and the trivial representation respectively. So the 
lowest $K-$type subquotient $\pi$ of $\mathrm{Ind}_{MN}^{E_8}(J_{E_7}(\nu_7;\nu_7) \otimes \mathbb{C}_{\tau(\eta)} \otimes \mathbf{1})$
contributes to $\widehat{G}^d$ whenever $2\la_L$ is regular. 

As for the middle parameter $2\nu_7 = [2,1,1,1,1,1,2]$, the same
calculations of signatures for the $K-$types in
($\omega_7\otimes \omega_7$) can be applied using \texttt{mathematica} for the linear algebra. The conclusion is that
$$J_{E_7}\left(\frac{1}{2}[2,1,1,1,1,1,2],\frac{1}{2}[2,1,1,1,1,1,2]\right)$$ 
has indefinite form on $V_{\mathfrak{e}_7 \cap
  \mathfrak{k}}(2\omega_7)$, a $K-$type that does not appear in $\texttt{adj}_7 \otimes
\texttt{adj}_7$. So the lowest $K-$type subquotient of
$\mathrm{Ind}_{MN}^{E_8}(J_{E_7}(\frac{1}{2}[2,1,1,1,1,1,2],$
$\frac{1}{2}[2,1,1,1,1,1,2]) \otimes \mathbb{C}_{\tau(\eta)} \otimes
\mathbf{1})$ cannot be in $\widehat{G}^d$ for any $k \geq 1$, since 
\begin{itemize}
\item[(a)] for $k=1$, $2\la_L = [2,1,1,1,1,1,2,-7]$ is singular, so it does not satisfy \eqref{eq-HP};
\item[(b)] for $k \geq 2$, $V_{\mathfrak{e}_7 \cap \mathfrak{k}}(2\omega_7)$ is $M-$bottom layer in the above induced module, so its subquotient is not unitary.
\end{itemize}

\begin{remark} \label{rmk-e7}
When $2\nu_7 = [1,1,1,1,1,1,1]$ or $[2,1,2,1,1,1,1]$, our calculations above imply that $2\la_L = 2\widetilde{\nu_7} + k\omega_8$ is
not regular integral for any $k \geq 1$. Note that the spherical modules $J_{E_7}(\nu_7,\nu_7)$ 
corresponding to these parameters are unipotent representations
attached to the nilpotent orbits $4A_1$ and $3A_1''$ respectively (see \cite[Section 6]{DW3} for details). It is shown in \cite{DW3} that these unipotent representations have no Dirac cohomology in $E_7$.

\end{remark}

\subsection{Type $E_6$} \label{subsec-e6} There are $2^4 = 16$ choices of Hermitian, spherical modules $J_{E_6}(\nu_6,\nu_6)$, namely
$$2\nu_6 = [a,b,c,d,c,a],\quad a,b,c,d \in \{1,2\}.$$ 
Among them, $J_{E_6}(\nu_6,\nu_6)$ has definite Hermitian form on the level of $\texttt{adj}_6 \otimes \texttt{adj}_6$ only when
$2\nu_6 = [1,1,1,1,1,1]$ or $[2,2,2,2,2,2]$. They are the unipotent representation corresponding to the model orbit $3A_1$ and the trivial representation respectively.

\smallskip
For the remaining $16-2 = 14$ choices of $\nu_6$, only the following can possibly make $2\la_L = 2\widetilde{\nu_6} + \omega_7 + l\omega_8$ regular integral:
\begin{equation} \label{eq-e6u}
\begin{aligned}
 2\nu_6 &= [2,1,2,1,2,2], \quad 2\la_L = [1,1,1,1,1,1,1,l-12];\\
 2\nu_6 &= [2,1,2,2,2,2], \quad 2\la_L = [1,1,2,1,1,1,1,l-14];\\
 2\nu_6 &= [2,2,2,1,2,2], \quad 2\la_L = [2,1,1,1,1,1,1,l-13]
\end{aligned}
\end{equation}

As in the Levi Type $D_6$ case, we extend the Levi subgroup to $M'$ of Type $E_7$ and consider the induced module
\begin{equation} \label{eq-inde7}
\mathrm{Ind}_{M'N'}^{E_8}\left(J_{E_7}(\nu_6^+;\nu_6^-) \otimes \mathbb{C}_{\tau'(\eta)} \otimes {\bf 1} \right)
\end{equation}
with $\nu_6^{\pm} = \la_L \pm \frac{1}{2}\omega_7$, whose lowest $K-$type is
equal to $V_{\frk}(\omega_7 + l\omega_8) = V_{\frk}(\eta)$. Note that $\eta = (0,0,0,0,0,1,l+1,l+2)$ in usual coordinates.

For our choices of $l$ above, the $(M' \cap K)-$type $V_{\mathfrak{k}}(\omega_1 + \omega_7+l\omega_8)$ in $J_{E_7}(\nu_6^+;\nu_6^-)$
is $M'-$bottom layer. More precisely, the $K-$type of highest weight $(0,0,0,0,0,1,l,l+3)$ is always dominant for $l \geq 1$. So it occurs with the same multiplicity and signature as the $(M' \cap K)-$type $V_{\mathfrak{m} \cap \mathfrak{k}}(\omega_1 + \omega_7)$ in $J_{E_7}(\nu_6^+;\nu_6^-)$.

\medskip
We now study \eqref{eq-inde7} for the three parameters of $\nu_6$ in \eqref{eq-e6u}.
For the first parameter, the module $J_{E_7}(\nu_6^+;\nu_6^-)$ in \eqref{eq-inde7} with $2\nu_6^+ = [1,1,1,1,1,1,1]$ is equal to the $19^{th}$ entry of \cite[Table 4]{DW3}, which is
a (nonspherical) unipotent representation attached to the nilpotent orbit $4A_1$ with nonzero Dirac cohomology. Consequently,
the first parameter yields a representation in the Dirac series of $E_8$ whenever $2\la_L$ satisfies \eqref{eq-HP}. 

As for the last two parameters, one can compute that $2\nu_6^+ = [2,1,1,1,1,1,1]$ and $[1,1,2,1,1,1,1]$ respectively. By \texttt{atlas}, the modules $J_{E_7}(\nu_6^+;\nu_6^-)$ for these two choices of $2\nu_6^+$ have indefinite signatures on the $(M' \cap K)-$types
$V_{\mathfrak{m}' \cap \mathfrak{k}}(\omega_7)$ and $V_{\mathfrak{m}' \cap \mathfrak{k}}(\omega_1 + \omega_7)$ (see Appendix). So the discussion in the previous paragraph implies that the irreducible lowest $K-$type
subquotients of \eqref{eq-inde7} corresponding to these two parameters are not unitary.
\bigskip

This finishes the proof of Proposition \ref{cor-e}.
\end{proof}


\medskip

\centerline{\scshape Funding} 
The authors would like to thank the referee for carefully reading the manuscript, along with many helpful suggestions.

Barbasch is supported by NSF grant 2000254.
Wong is supported by the National Natural Science Foundation of China (no. 12341101, 12371033) and the 
Shenzhen Science and Technology Innovation Committee grant
(no. 20220818094918001).

\section*{Appendix - Some atlas calculations}



In this section, we study the two non-spherical modules $J_{E_7}(\nu_6^+;\nu_6^-)$ appearing in \eqref{eq-inde7} with lowest $K$-type $V_{\mathfrak{e}_7 \cap \mathfrak{k}}(\omega_7)$. The authors would like to thank Chao-Ping Dong for carrying out these calculations.

\medskip
For $2\nu_6^+ = [2,1,1,1,1,1,1]$, one has:
\small
\begin{verbatim}
atlas> set q1=parameter(x,[0,1,0,-1,1,-1,1,0,0,0,0,0,0,0],[2,0,1,2,0,2,0,0,0,0,0,0,0,0])
atlas> set p1=first_param(finalize(q1))
atlas> infinitesimal_character(p1)
Value: [2,1,1,1,1,1,1,2,1,1,1,1,1,1]/2
atlas> print_sig_irr_long(p1,x,height(p1)+60)
sig  x  lambda                         hw                            dim    height
1    0  [0,0,0,0,0,0,0,0,0,0,0,0,0,1]  [1,1,1,1,1,1,1,1,1,1,1,1,1,2]  56     27
1    0  [0,0,0,0,0,0,0,0,1,0,0,0,0,0]  [1,1,1,1,1,1,1,1,2,1,1,1,1,1]  912    49
2+s  0  [0,0,0,0,0,0,0,1,0,0,0,0,0,1]  [1,1,1,1,1,1,1,2,1,1,1,1,1,2]  6480   61
2    0  [0,0,0,0,0,0,0,0,0,0,0,1,0,0]  [1,1,1,1,1,1,1,1,1,1,1,2,1,1]  27664  75
2+s  0  [0,0,0,0,0,0,0,0,0,0,0,0,1,1]  [1,1,1,1,1,1,1,1,1,1,1,1,2,2]  51072  79
1+s  0  [0,0,0,0,0,0,1,0,0,0,0,0,0,2]  [1,1,1,1,1,1,0,1,1,1,1,1,1,3]  24320  81
2+s  0  [0,0,0,0,0,0,0,1,1,0,0,0,0,0]  [1,1,1,1,1,1,1,2,2,1,1,1,1,1]  86184  83
\end{verbatim}

\medskip
\normalsize The case of $2\nu_6^+ = [2,1,1,1,1,1,1]$ is as follows:
\small
\begin{verbatim}
atlas> set q2=parameter(x,[0,1,0,-1,1,-1,1,0,0,0,0,0,0,0],[1,0,2,2,0,2,0,0,0,0,0,0,0,0])
atlas> set p2=first_param(finalize(q2))
atlas> infinitesimal_character(p2)
Value: [1,1,2,1,1,1,1,1,1,2,1,1,1,1]/2
atlas> print_sig_irr_long(p2,x,height(p2)+60)
sig   x  lambda                         hw                             dim    height
1     0  [0,0,0,0,0,0,0,0,0,0,0,0,0,1]  [1,1,1,1,1,1,1,1,1,1,1,1,1,2]  56     27
1     0  [0,0,0,0,0,0,0,0,1,0,0,0,0,0]  [1,1,1,1,1,1,1,1,2,1,1,1,1,1]  912    49
2+s   0  [0,0,0,0,0,0,0,1,0,0,0,0,0,1]  [1,1,1,1,1,1,1,2,1,1,1,1,1,2]  6480   61
2+s   0  [0,0,0,0,0,0,0,0,0,0,0,1,0,0]  [1,1,1,1,1,1,1,1,1,1,1,2,1,1]  27664  75
2+2s  0  [0,0,0,0,0,0,0,0,0,0,0,0,1,1]  [1,1,1,1,1,1,1,1,1,1,1,1,2,2]  51072  79
1+s   0  [0,0,0,0,0,0,1,0,0,0,0,0,0,2]  [1,1,1,1,1,1,0,1,1,1,1,1,1,3]  24320  81
3+2s  0  [0,0,0,0,0,0,0,1,1,0,0,0,0,0]  [1,1,1,1,1,1,1,2,2,1,1,1,1,1]  86184  83
\end{verbatim}

\normalsize By looking at the third line of the above calculations, there exists a copy of $M' \cap K$-type $V_{\mathfrak{m}' \cap \mathfrak{k}}(\omega_1 + \omega_7)$ having different signature from the lowest $M' \cap K$-type $V_{\mathfrak{m}' \cap \mathfrak{k}}(\omega_7)$ for both modules.

\end{document}